\documentclass[11pt,oneside,english]{amsart}
\usepackage[T1]{fontenc}
\usepackage[latin1]{inputenc}
\usepackage[a4paper]{geometry}
\usepackage{textcomp}
\usepackage{amsthm}
\usepackage{amstext}
\usepackage{amssymb}
\usepackage{esint}

\makeatletter

\let\SF@@footnote\footnote
\def\footnote{\ifx\protect\@typeset@protect
    \expandafter\SF@@footnote
  \else
    \expandafter\SF@gobble@opt
  \fi
}
\expandafter\def\csname SF@gobble@opt \endcsname{\@ifnextchar[%]
  \SF@gobble@twobracket
  \@gobble
}
\edef\SF@gobble@opt{\noexpand\protect
  \expandafter\noexpand\csname SF@gobble@opt \endcsname}
\def\SF@gobble@twobracket[#1]#2{}

\numberwithin{equation}{section}
\numberwithin{figure}{section}
\theoremstyle{plain}
\newtheorem{thm}{Theorem}[section]
  \theoremstyle{plain}
  \newtheorem{cor}[thm]{Corollary}
  \theoremstyle{plain}
  \newtheorem{prop}[thm]{Proposition}
  \theoremstyle{remark}
  \newtheorem{rem}[thm]{Remark}
  \theoremstyle{plain}
  \newtheorem{lem}[thm]{Lemma}
  \theoremstyle{definition}
  \newtheorem{defn}[thm]{Definition}

\usepackage{amsthm}

\usepackage{amssymb}

\usepackage{amssymb}

\usepackage{amssymb}

\def\makebbb#1{
    \expandafter\gdef\csname#1\endcsname{
        \ensuremath{\Bbb{#1}}}
}
\makebbb{R}
\makebbb{N}
\makebbb{Z}
\makebbb{C}
\makebbb{D}
\makebbb{E}
\makebbb{H}
\makebbb{P}
\makebbb{B}
\makebbb{K}
\makebbb{E}

\usepackage{babel}

\usepackage{babel}

\usepackage{babel}

\usepackage{babel}

\makeatother

\usepackage{babel}

\begin{document}

\title{determinantal point processes and fermions on complex manifolds:
large deviations and bosonization}

\author{Robert J.Berman}

\email{robertb@chalmers.se}

\curraddr{Mathematical Sciences - Chalmers University of Technology and University
of Gothenburg - SE-412 96 Gothenburg, Sweden }
\begin{abstract}
We study determinantal random point processes on a compact complex
manifold $X$ associated to an Hermitian metric on a line bundle over
$X$ and a probability measure on $X.$ Physically, this setup describes
a free fermion gas on $X$ subject to a $U(1)-$ gauge field and when
$X$ is the Riemann sphere it specializes to various random matrix
ensembles. It is shown that, in the many particle limit, the empirical
random measures on $X$ converge exponentially towards the deterministic
pluripotential equilibrium measure, defined in terms of the Monge-Ampere
operator of complex pluripotential theory. More precisely, a large
deviation principle (LDP) is established with a good rate functional.
We also express the LDP in terms of the Ray-Singer analytic torsion
and the exponentially small eigenvalues of $\bar{\partial}-$Laplacians.
This can be seen as an effective bosonization formula, generalizing
the previously known formula in the Riemann surface case to higher
dimensions and the paper is concluded with a heuristic quantum field
theory intepretation of the resulting effective boson-fermion correspondence. 

\tableofcontents{}
\end{abstract}
\maketitle

\section{Introduction}

In this paper we study a natural class of determinantal random point
processes \cite{m,h-k-p} defined on a compact complex manifold $X.$
These processes are induced by the choice of a polarization of $X,$
i.e an ample line bundle $L$ over $X$ and we will be concerned with
their many particle limit. When $X$ is the Riemann sphere this geometric
setup contains the extensively studied (Hermitian, unitary and normal)
random matrix ensembles. Suitable higher dimensional choices of polarized
$X$ give rise to multivariate orthogonal polynomial ensembles, as
well as their trigonometric and spherical counterparts (see section
\ref{sec:Examples}). On a general complex manifold the point processes
represent, from a physical point of view, a gas of (chiral/spin polarized)
fermions coupled to a gauge field on $X$ with gauge group $U(1)$
(see section \ref{sec:Relation-to-bosonization}). In broad terms
the main aim of this paper is to decribe the many particle limit in
terms of global pluripotential theory and relate it to the notion
of bosonization (boson-fermion correspondences) in the physics litterature,
previously known only in the Riemann surface case \cite{st}. In the
companion paper \cite{berm3} central limit theorems and universality
of correlation functions were obtained. Here we will be concerned
with the large deviation regime, establishing a large deviation principle
(LDP) for the empirical measure of the point process.

Another concrete motivation for the LDP comes from probabilistic methods
for locating nearly optimal nodes for interpolating polynomials of
large degree on a given set $K.$ Such optimal nodes are commenly
defined as configurations of points $(x_{1},...,x_{N})$ maximizing
the density of the probability measure \ref{eq:def of det process intro}
below \cite{sl-w} (i.e. as configurations of Fekete points on $K$
\cite{b-b-w}). 

It may be illluminating to first formulate the main results to be
obtained in an informal manner. Consider a gas of $N$ identical particles
on $X$ decribed by a symmetric probability measure $\mu^{(N)}$ on
$X^{N}$ with density $\rho^{(N)}(x_{1},...,x_{N})$ wrt a fixed volume
form $dV$ on $X.$ The theory of large deviations \cite{d-z} allows
one to give a meaning to the statement that a $\mu^{(N)}$ is {}``exponentially
concentrated on a deterministic macroscopic measure $\mu_{eq}$ (the
equilibrium measure) with a rate functional $H(\mu).$ The idea is
to think of the large $N-$limit of the $N-$particle space $X^{N}$
of configurations of points $(x_{1},...,x_{N})$ ({}``microstates'')
as being approximated by a space of {}``macrostates'', which is
the space $\mathcal{P}(X)$ of all probability measures on $X:$ \[
X^{N}\sim\mathcal{P}(X),\]
 as $N\rightarrow\infty.$ The exponential concentration referred
to above may then be informally written as \begin{equation}
\mu^{(N)}:=\rho^{(N)}(x_{1},...,x_{N})dV(x_{1})\otimes\cdots\otimes dV(x_{N})\sim e^{-a_{N}H(\mu)}\mathcal{D}\mu,\label{eq:large dev heur}\end{equation}
 where $\mathcal{D}\mu$ denotes a (formal) probability measure on
the infinite dimensional space $\mathcal{P}(X)$ and where the sequence
of numbers $a_{N}$ is called the speed (or rate) which is usually
a power of $N).$ Exponential concentration around $\mu_{eq}$ appears
when $H(\mu)\geq0$ with $H(\mu)=0$ precisely when $\mu=\mu_{eq}.$
The {}``change of variables'' from $X^{N}$ to $\mathcal{P}(X)$
is made precise by using the embedding \[
j_{N}:\, X^{N}\rightarrow\mathcal{P}(X),\,\,\, j_{N}(x,...,x_{N}):=\frac{1}{N}\sum_{i}\delta_{x_{i}}\]
and then pushing forward the probability measure $\mu^{(N)}$ on $X^{N}$
to $\mathcal{P}(X)$ with the map $j_{N},$ giving a probability measure
$(j_{N})_{*}\mu^{(N)}$ on $\mathcal{P}(X).$ The meaning of \ref{eq:large dev heur},
in the sense of large deviations, is then 

\[
-\inf_{\mu\in\mathcal{G}}H(\mu)\leq\lim_{N\rightarrow\infty}\frac{1}{a_{N}}\log\int_{\mathcal{G}}(j_{N})_{*}\mu^{(N)}\leq\lim_{N\rightarrow\infty}\frac{1}{a_{N}}\log\int_{\overline{\mathcal{G}}}(j_{N})_{*}\mu^{(N)}\leq-\inf_{\mu\in\overline{\mathcal{G}}}H(\mu)\]
 for any open set $\mathcal{G}$ in $\mathcal{P}(X).$ 

In the present setting the gas can be represented by spin polarized
free fermions on $X$ coupled to a $U(1)-$gauge field $A$ and we
will write $\omega=\frac{i}{2\pi}F_{A}$ for the corresponding magnetic
two-form, normalized so that $[\omega]\in H^{2}(X,\mbox{\Z).}$ The
probability density \[
\rho^{(N)}(x_{1},...x_{N})=\frac{1}{\mathcal{Z}_{N}}\left\Vert \det\Psi\right\Vert ^{2}(x_{1},...x_{N})\]
 is the normalized Slater determinant representing the maximally filled
$N-$particle ground state. We will consider the limit of increasing
field strength, i.e. $F_{A}$ is replaced by $kF_{A}$ with $k\rightarrow\infty$
so that $N=N_{k}=Vk^{n}+o(k^{n})\rightarrow\infty,$ where $n=\dim_{\C}X$
(assuming that the underlying line bundle $L$ is ample which also
induces the spin polarization). It will be shown that an LDP of the
form \ref{eq:large dev heur} holds at a speed $Vk^{n+1}$ with a
rate functional $H(\mu)$ that may be decomposed as \begin{equation}
H(\mu)=E_{\omega}(\mu)-C\label{eq:rate function intro}\end{equation}
 where $E_{\omega}(\mu)$ is the \emph{pluricomplex energy of $\mu$}
recently introduced in \cite{bbgz} and the constant $C=C(K,\omega)$
is the pluricomplex capacity $C(K,\omega)$ (in logarithmic notation).
In the case of a Riemann surface $E_{\omega}$ is nothing but the
Dirichlet energy of a unit charge distribution $\mu$ subject to the
neutralizing exterior charge $\omega:$ \[
E_{\omega}(\mu)=\frac{1}{2}\int_{X}d\phi_{\mu}\wedge d^{c}\phi_{\mu},\,\,\,\, dd^{c}\phi_{\mu}=\mu-\omega\]
 wheere $dd^{c}=\frac{i}{2\pi}\partial\bar{\partial}$ (see section
\ref{sub:Notation-and-general}). In higher dimensions $E_{\omega}(\mu)$
may be expressed explictely in terms of the potential $\phi_{\mu}$
of $\mu$ obtained by solving the highly non-linear inhomogeonous
Monge-Ampère equation. We will recall the necessary background from
pluripotential theory in section \ref{sec:The-Monge-Ampere-operator,}.
For the moment we just point out that when $\mu$ is a volume form
the existence of a smooth potential $\phi_{\mu}$ was shown by Yau
in his celebrated solution of the Calabi conjecture \cite{y}. The
functional $E_{\omega}$ is highly non-linear when $n>1$ - for example
when $n=2$ one has \[
E_{\omega}(\mu)=\frac{1}{2}\int_{X}d\phi_{\mu}\wedge d^{c}\phi_{\mu}\wedge\omega+\frac{1}{3}\int_{X}d\phi_{\mu}\wedge d^{c}\phi\wedge(dd^{c}\phi_{\mu})\]
 and in general it is of degree $n$ in the potential $\phi_{\mu}.$ 

The relation to bosonization appears when using a different normalization
of the Slater determinant so that it becomes (at least formally) the
$N-$point function $\left\langle \left\Vert \Psi(x_{1})\right\Vert ^{2}\cdots\left\Vert \Psi(x_{N})\right\Vert ^{2}\right\rangle $
of a fermionic quantum field theory on $X$ defined by the corresponding
(massless) Dirac action. Mathematically, this amounts to multiplying
$\left\Vert \det\Psi\right\Vert ^{2}(x_{1},...x_{N})$ with the analytic
torsion. We will then show that this has the effect of cancelling
the constant $C$ in the expression \ref{eq:rate function intro}
for the rate function. In the final section of the paper we interpretate
the resulting LDP as an effective bosonization: \begin{equation}
\left\langle \left\Vert \Psi(x_{1})\right\Vert ^{2}\cdots\left\Vert \Psi(x_{N})\right\Vert ^{2}\right\rangle \sim\left\langle e^{i\phi(x_{1})}\cdots e^{i\phi(x_{N})}\right\rangle ,\label{eq:boson ans}\end{equation}
in the large $N-$limit where the right hand side is expressed in
terms of the (formal) quantum field theory for a bosonic field $\phi$
with an explicit action, coinciding, up to scaling, with a secondary
Bott-Chern class (which in physics terminology is an example of a
{}``higher-derivative action''). In the physics litterature bosonization
is a well-known phenomena in $1+1$ real dimensions (i.e. $n=1)$
\cite{st} and its present higher dimensional incarnation appears
to be somewhat surprising (however see \cite{b-l-q,f-g-m,sch} for
possibly related results). But one explanation, appart from the fact
that it only holds effectively/asymptotically, may be the extra condition
imposed by the complex/holomorphic structure when $n>1:$ $X$ is
a complex manifold and $F_{A}$ is assumed to be a $(1,1)-$ form,
i.e. it determines a holomorphic structure on the unerlying line bundle
$L$. 

Before turning to the precise formulation of the geometric setup we
emphasize that we will be considering a more general setting where
the volume form $dV$ on $X,$ used above, is replaced with a suitable
measure $\nu$ on $X$ supported on a compact set $K.$ One of the
main points of considering this more general setting is that it allows
one to treat completely\emph{ real} situations where $X$ appears
as a complexification of $K.$ For example, $K$ may be taken to be
the real $n$-sphere $S^{n}$ or the $n-$torus $T^{n}.$

\subsection{The geometric setup}

Let $L\rightarrow X$ be an ample holomorphic line bundle over a compact
complex manifold $X$ of dimension $n.$ We will denote by $H^{0}(X,L)$
the $N-$dimensional vector space of all global holomorphic sections
of $L.$ Given the geometric data $(\nu,\left\Vert \cdot\right\Vert )$
consisting of a probability measure $\nu$ on $X$ and an Hermitian
metric $\left\Vert \cdot\right\Vert $ on $L$ which is continuous
on the support $K$ of $\nu,$ one obtains an associated probability
measure $\mu^{(N)}$ on the $N-$fold product $X^{N}$ defined as
\begin{equation}
\mu^{(N)}:=\frac{1}{\mathcal{Z}_{N}}\left\Vert \det\Psi\right\Vert ^{2}(x_{1},...x_{N})\nu(x_{1})\otimes\cdots\otimes\nu(x_{N})\label{eq:def of det process intro}\end{equation}
where $\det\Psi$ is any holomorphic section of the pulled-back line
bundle $L^{\boxtimes N}$ over $X^{N}$ representing the complex line
$\Lambda^{N}H^{0}(X,L)$ and $\mathcal{Z}_{N}$ is the normalizing
constant. Concretely, we may write $\det\Psi$ as a Slater determinant:
\begin{equation}
(\det\Psi)(x_{1},...,x_{N_{k}})=\det(\Psi_{i}(x_{j}))\label{eq:slater intro}\end{equation}
for a given base $(\Psi_{i})_{=1}^{N}$ of $H^{0}(X,L).$ We will
denote $\frac{i}{2\pi}$ times the curvature two-form (current) of
the metric on $L$ by $\omega.$ It will be convenient to take the
pair $(\omega,\nu),$ which will refer to as a\emph{ weighted measure},
as the given geometric data. The\emph{ empirical measure} of the ensemble
above is the following random measure: \begin{equation}
(x_{1},...,x_{N})\mapsto\delta_{N}:=\sum_{i=1}^{N}\delta_{x_{i}}\label{eq:intro random measure}\end{equation}
which associates to any $N-$particle configuration $(x_{1},...,x_{N})$
the sum of the delta measures on the corresponding points in $X.$
In probabilistic terms this setting hence defines a \emph{determinantal
random point process on $X$ with $N$ particles} \cite{h-k-p,j2}. 

If the correponding $L^{2}-$norm on $H^{0}(X,L)$ \[
\left\Vert \Psi\right\Vert _{X}^{2}=\left\langle \Psi,\Psi\right\rangle _{X}:=\int_{X}\left\Vert \Psi(x)\right\Vert ^{2}d\nu(x)\]
is non-degenerate (which will always be the case in this paper) then
the probability measure $\mu^{(N)}$ on $X^{N}$ may be expressed
as a determinant of the\emph{ Bergman kernel} of the Hilbert space
$(H^{0}(X,L),\left\Vert \cdot\right\Vert _{X}),$ i.e. the integral
kernel of the corresponding orthogonal projection. But one virtue
of the definition \ref{eq:def of det process intro} is that it admits
a natural generalization to\emph{ $\beta-$ensembles,} obtained by
replacing the power $2$ with a general real positive power $\beta$
(which plays the the role of \emph{inverse temperature} from the point
of view of statistical mechanics - see section \ref{sub:Proof-ldp-non-cpt}
for such generalizations). Replacing $L$ with its $k$ th tensor
power, which we will write in additive notation as $kL,$ yields,a
sequence of point processes on $X$ of an increasing number $N_{k}$
of particles. We will be concerned with the asymptotic situation when
$k\rightarrow\infty.$ This correspondence to a large $N-$limit of
many particles, since \[
N_{k}:=\dim H^{0}(X,kL)=Vk^{n}+o(k^{n})\]
where the constant $V$ is, by definition, the \emph{volume} of $L.$ 

Of course, we can also view point processes above as defined on the
support $K$ of the measure $\nu.$ There is also a slight variant
of the setting above where $K$ may be taken as the non-compact sets
$\C^{n}$ or $\R^{n}$ (which may be identified with subsets of $X:=\P^{n}$
the complex projective space) - see section \ref{sub:The-non-compact-setting}.

\subsection{Statement of the main results}

\subsubsection{A general large deviation principle}

As is well-known the density of the one-point correlation measure,
i.e of the expectation $\E(\delta_{N_{k}})$ of the emprical measure,
is precisely the point-wise norm of the corresponding Bergman kernel
on the diagonal. In the case when the curvature form $\omega$ on
$L$ is smooth and positive and the measure $\nu$ is a volume form
on $X$ it then follows from well-known Bergman kernel asymptotics
that \[
\E(\frac{\delta_{N_{k}}}{N_{k}})\rightarrow\frac{1}{V}\frac{\omega^{n}}{n!}\]
 weakly as $k\rightarrow\infty$ where $\E$ denotes the expectation
with respect to the determinantal ensemble $(X^{N},\mu^{(N)})$ (in
fact there is a complete asymptotic expansion in powers of $1/k$
as first shown by Catlin and Zeldich; see the survey \cite{z} and
references therein and also \cite{d-k} for a path integral approach).
For a curvature form $\omega$ which is smooth, but not necesserily
semi-positive, it was shown in \cite{berm1} that the the previous
convergence still holds if the right hand side above is replaced by
the pluripotential \emph{equilibrium measure }$\mu_{eq}$ on $X$
associated to $\omega$ which may be written as \[
\mu_{eq}=1_{D}\frac{1}{V}\frac{\omega^{n}}{n!}\]
where $D$ is a certain compact subset of $X.$ In the case of a Riemann
surface (i.e. $n=1)$ the set $D$ may be obtained by solving a free
boundary value problem for the Laplace operator (in the physics litterature
the set $D$ appears as a limiting Coulomb gas plasma/fluid in the
context of the Quantum Hall effect, as well as an eigenvalue droplet
in the normal random matrix model - see \cite{za} and references
therein). For a general dimension $n$ the set $D$ is obtained similarly
but using the Monge-Ampère operator, which is fully non-linear (see
section \ref{sec:The-Monge-Ampere-operator,}). For general geometric
data $(\omega,\nu),$ with $K$ denoting the compact set obtained
as the support of $\nu,$ the convergence towards the equilibrium
measure $\mu_{eq}$ (in the sense of pluripotential theory) associated
to the weighted set compact set $(\omega,K)$ was shown to hold very
recently in \cite{b-b-w} under very weak regularity assumptions on
the measure $\nu$ its support $K.$ 

Our first main result shows that the convergence of the empirical
measure is in fact \emph{exponential} in probability. More precisely,
we have the following \emph{Large Deviation Principle} \emph{(LDP})
for the\emph{ laws} of the normalized empirical measures, i.e. for
the push forward of the probability measure $\mu^{(N_{k})}$ on $X^{N}$
to the space $\mathcal{P}(K)$ of all probability measures on $K,$
under the normalized map $\delta_{N}/N.$
\begin{thm}
\label{thm:intro large dev}(LDP) Let $(\nu,\omega)$ be a weighted
measure and assume that the the measure $\nu$ has the Bernstein-Markov
property and that its support $K$ is non-pluripolar. Then the laws
of the normalized emipirical measures of the determinantal point process
\ref{eq:def of det process intro} satisfy a \emph{large deviation
principle} (LDP) with a \emph{good rate functional $H(=H_{(K,\omega)})$}
and speed $Vk^{n+1}(\sim kN_{k}).$ On the space $\mathcal{P}(K)$
the rate functional \emph{$H(\mu)$} is minimized (and vanishes) precisely
on the pluripotential equilibrium measure $\mu_{eq}(=\mu_{(K,\omega)}).$
More generally, the upper bound corresponding to the LDP holds for
\emph{any} $(\nu,\omega).$
\end{thm}
As is shown in section \ref{sub:Proof-ldp-non-cpt} and \ref{sub:A-generalized-LDP beta},
respectively, the LDP in the previous theorem can be adapted to two
other general settings:
\begin{itemize}
\item The LDP holds in a setting where $K$ is allowed to be non-compact
\item The LDP holds for $\beta-$ensembles (section \ref{sub:A-generalized-LDP beta})
\end{itemize}
The proof in the non-compact setting proceeds by reducing to the previous
case when $K$ is compact. In particular we obtain the following LDP
for the Vandermonde determinant $\Delta^{(N_{k})}$ obtained by taken
the base $(\Psi_{i})_{=1}^{N_{k}}$ in \ref{eq:slater intro} to be
multinomials in $\C^{n}$ of total degree at most $k$ (see Cor \ref{cor:ldp for vanderm with rate fu}
and section \ref{sub:Remarks-on-normalizations}  for a description
of the corresponding rate and energy functionals). 
\begin{cor}
\label{cor:vanderm intro}Let $K=\R^{n}$ (or $K=\C^{n})$ and let
$\nu$ be the Euclidean measure on $K.$ Assume that $\Phi$ is a
continuous function on $K$ with super logarithmic growth at infinity
(see \ref{eq:ass on growth of phi} ). Then the push-forward $\Gamma_{k}$
of the weighted Vandermonde measure \[
\tilde{\mu}_{\Phi}^{(N_{k})}:=|\Delta^{(N_{k})}(z_{1},....,z_{N_{k}})|^{2}e^{-k\Phi(z_{1})}\cdots e^{-k\Phi(z_{N_{k}})}\nu^{\otimes N_{k}}\]
 under the normalized map $\delta_{N_{k}}/N_{k}$ (formula \ref{eq:intro random measure})
satisfies an LDP at a speed $k^{n+1}$ and with a good rate functional,
which in the case $n=1$ coincides the the weighted logarithmic energy
\cite{s-t}. More generally, the LDP holds (at a speed $\beta_{k}k^{n+1}$)
when the density of $\tilde{\mu}_{\Phi}^{(N_{k})}$ is raised to a
positive power $\beta_{k}$ as long as $\beta_{k}\leq C$ and $\beta_{k}k\rightarrow\infty.$
\end{cor}
As another corollary (see section \ref{sub:Applications-to-sections}
) we obtain an LDP for ensembles defined by holomorphic sections vanishing
to high order along a given hypersurface in $X.$ Physically, this
allows one to consider situations where the fermion ground-state has
a filling fraction stricly below $1,$ as in the the fractional Quantum
Hall effect, which is well-known in the case when $n=1.$ We also
point out some relations to Laplacian growth \cite{h-m}.

It should be stressed that the assumptions in Theorem \ref{thm:intro large dev}
are very weak and they are satisfied in geometrically natural situations.
For example, $\nu$ may be taken to be defined by integrating against
a volume form on a smooth domain $K$ \cite{b-b-w}. The measure $\nu$
can also be taken to be defined by the Haussdorf measure of a set
$K$ which is either a real submanifold of real dimension $2n-1$
or a real algebraic variety of real dimension $n$ (see section \ref{sec:Examples}). 

Given a function $\phi$ on $X$ we denote by 

\begin{equation}
\delta_{N_{k},\lambda}(\phi):=\mbox{Prob}\left\{ \left|\frac{1}{N_{k}}(\phi(x_{1})+....+\phi(x_{N_{k}}))-\int_{X}\mu_{eq}\phi\right|>\lambda\right\} \label{eq:def of tail}\end{equation}
the\emph{ tail of the linear statistic} $\phi(x_{1})+....+\phi(x_{N_{k}}).$ 
\begin{cor}
\label{cor:non-sharp tail intro}Let $\phi$ be a continuous function
on $X.$ 
\begin{itemize}
\item If the support $K$ of the measure $\nu$ is not pluripolar, then
the tail \ref{eq:def of tail} satisifes $\delta_{k,\epsilon}(\phi)\leq e^{-Ck^{(n+1)}}$
for some positive constant $C$ depending on $\phi$ and $\epsilon.$ 
\item In the case when $X$ is a Riemann surface, $K=X$ and the curvature
current $\omega$ of the metric on $L$ is semi-positive (so that
$\mu_{eq}=\omega)$ the following more precise estimate holds: \begin{equation}
\delta_{N_{k},\lambda}(\phi)\leq2\exp(-N_{k}^{2}\left(\frac{2V\lambda^{2}}{\left\Vert d\phi\right\Vert _{X}^{2}}(1+o(1))\right))\label{eq:exp conc general riemann s}\end{equation}
 where the error term $o(1)$ denotes a sequence tending to zero as
$k\rightarrow\infty$ (but depending on $\phi$). 
\end{itemize}
\end{cor}

\subsubsection{\label{sub:Analytic-torsion,-exponentially}Analytic torsion, exponentially
small eigenvalues and effective bosonization}

Now we fix a smooth Hermitian metric $h_{X}$ on $X$ and take the
measure $\nu=dV$ above to be its volume form (so that $K=X)$. We
will also assume that the Hermitian metric on $L$ is smooth - the
most interesting case will be when its curvature $\omega$ is not
semi-positive. We will then obtain a reformulation of the LDP in Theorem
\ref{thm:intro large dev} involving the regularized determinants
of the associated $\overline{\partial}-$Laplacians $\Delta_{\bar{\partial}}^{0,q}$
acting on the space of $(0,q)-$forms with values in $kL.$ 
\begin{thm}
(E\label{thm:(Effective-bosonization)-Let}ffective bosonization)
Let $L\rightarrow X$ be a line bundle over a compact complex manifold
$X$ and equip $L$ and $X$ with smooth metrics as above. Then the
push forward of the measures \begin{equation}
\prod_{q=1}^{n}(\det\Delta_{\bar{\partial}}^{0,q})^{(-1)^{q+1}q}\left\Vert \det\Psi\right\Vert ^{2}(x_{1},...x_{N_{k}})dV(x_{1})\otimes\cdots\otimes dV(x_{N})\label{eq:fermion measure intro}\end{equation}
on $X^{N_{k}}$ under the map $\delta_{N}/N$ where $\delta_{N}$
is the empricial measure \textup{\ref{eq:intro random measure}} satisfy
a LDP with a good rate functional $E_{\omega}(\mu)$ (the pluricomplex
energy wrt $\omega).$ In particular, \begin{equation}
\frac{1}{k^{n+1}}\sum_{q=1}^{n}q(-1)^{q}\log\det((\Delta_{\bar{\partial}}^{0,q})\rightarrow\inf_{\mu\in\mathcal{P}(X)}E_{\omega}(X)\label{eq:conv of anal tor intro}\end{equation}
as $k\rightarrow\infty.$
\end{thm}
The new information compared to Theorem \ref{thm:intro large dev}
is the limit \ref{eq:conv of anal tor intro} for the scaled logarithms
of the \emph{Ray-Singer analytic torsion.} The determinants above
are defined by using zeta-function regularization of the product of
the\emph{ positive} eigenvalues of $\Delta_{\bar{\partial}}^{0,q}.$
In the case when $h_{X}$ is a Kähler metric an equivalent form of
the convergence \ref{eq:conv of anal tor intro} was shown in \cite{b-b},
using among other things the deep anomaly formula for the Quillen
metric of Bismut-Gillet-Soulé for the determinant line $\mbox{DET\ensuremath{(L)}associated to the \ensuremath{\overline{\partial}-}complex. }$
The proof in the general non-Kähler case is given in \cite{berm4}.
The main point of the proof is to show that the same statement also
holds when $\det\Delta_{\bar{\partial}}^{0,q}$ is replaced by the
product of all positive \emph{exponentially small} eigenvalues. 

The previous theorem can be seen as an effective (i.e. asymptotic)
generalization to higher dimensions of the\emph{ bosonization formula}
on a Riemann surface, saying that\begin{equation}
\det\Delta_{\bar{\partial}}\left\Vert \det\Psi\right\Vert ^{2}(x_{1},...x_{N})=C_{N,g}\exp(+\left(\frac{1}{2}\sum_{i\neq j}G(x_{i},x_{j})+r(x_{1},....,x_{N})\right))\label{eq:boson formula}\end{equation}
where $G$ is the Green function of the Laplacian defined wrt the
Arakelov metric $\omega$ on $X$ (when $g>0).$ The term $r$ vanishes
for genus $g=0.$ For $g>0$ it may be expressed in terms of the Riemann
theta function on the Jacobian torus of the Riemann surface $X.$
This formula was first obtained by a heuristic argument in \cite{b-v-},
using a fermion-boson ansatz and then rigourosly proved using properties
of Quillen metrics and complex algebraic geometry (see the appendix
in \cite{b-v-}). See also \cite{v-v} for another heuristic argument.
However, an explicit expression for the factor $C_{N,g}$ was determined
only very recently \cite{w}. Combining Theorem \ref{thm:(Effective-bosonization)-Let}
with the exact formula \ref{eq:boson formula} shows that both $C_{N,g}$
and $r$ are negiligable in the large $N-$limit (in particular it
follows that $\log C_{N,g}=o(N^{2}),$ which is consistent, as it
must, with the explicit formula found in \cite{w}).

\subsection{Relations to previous results}

In the case when $n=1$ and $K=\R$ the LDP in Cor \ref{cor:vanderm intro}
(for $\beta_{k}\equiv\beta)$ was first obtained by Ben Arous-Guinnet
\cite{b-g} (see also Ben Arous-Zeitouni \cite{b-z} for the case
$K=\C$ and $\Phi(z)=|z|^{2})$. The result in \cite{b-g} was formulated
in terms of the standard Hermitian random matrix ensemble, building
on previous work by Voiculescu on free probability theory. In particular,
these results imply the convergence of the free energies $k^{-1}N_{k}^{-1}\log\int\tilde{\mu}_{\Phi}^{(N_{k})}$
previously established by Johansson \cite{j} when $K=\R$ using a
large deviation type upper bound (see also Hedenmalm-Makarov \cite{h-m}
for the case $K=\C)$. An elegent potential theoretic derivation of
Johansson's bound for general Bernstein-Markov measures supported
on $\C$ was introduced by Bloom-Levenberg \cite{b-l0} (whose global
pluripotential version also plays an important role in the present
paper). The LDP in Theorem \ref{thm:intro large dev} in the case
when $K\subset\P^{1}$ for $K$ an arbitary non-polar compact set
is contained in the analysis in the recent work \cite{z-z}, where
zeroes of random polynomials are considered. In all these works the
starting point is the explicit expression for $\log\left\Vert \det\Psi\right\Vert ^{2}(x_{1},...x_{N})$
as a sum of Green functions $G(x_{i},x_{j})$ of the Laplace operator
(which is the simple genus $0$ case of the bosonization formula \ref{eq:boson formula})
and it is shown that the LDP can be expressed in terms of the limiting
Green energy of a measure $\mu.$ If $G$ were bounded then the LDP
would follow from general asymptotics for Laplace type integrals,
but the non-boundedness leads to highly non-trivial analytic issues.
The most subtle point in the proof is the lower bound, which is handled
using a decomposition argument of the measure $\mu.$ The method of
proof in the present paper is completely different, as there is no
useful analogoue of the Green function when $n>1,$ which is a reflection
of the fact that the Monge-Ampère operator is fully non-linear. 

The present paper is a substantially revised and extended version
of the first preprint that appeared on ArXiv (which only contained
Theorem \ref{thm:intro large dev}). The main new features are that
(1) a second proof of the LDP is added which uses the abstract Gärtner-Ellis
theorem (2) the LDP has been extended to a non-compact setting and
to $\beta-$ensembles and (3) the relation of the LDP to bosonization
is explained and explored. Since the first version of the paper, results
equivalent to the LDP in theorem \ref{thm:intro large dev} in the
case of $\P^{n}$ have been obtained by Bloom-Levenberg \cite{b-l-1}.
Their proof of the lower bound in the LDP is different than the ones
in the present paper (but it also uses \cite{b-b}) . Moreover, an
alternative proof of the LDP in the Riemann surface is also contained
in the arguments in Zelditch's paper \cite{z2}, which rely on the
explicit bosonization formula \ref{eq:boson formula}.

\subsection{Organization}

In section \ref{sec:Examples} we consider concrete examples obtained
by specializing the general setting to get ensembles of multivariate
polynomials on $\C^{n}$ etc. In section \ref{sec:The-Monge-Ampere-operator,}
we introduce the global pluripotential theory needed to define and
study the rate function of the LDP:s. Then in section \ref{sec:Large-deviations}
two proofs of the LDP in Theorem \ref{thm:intro large dev} are given.
The first one is direct, while the second one involves the abstract
Ellis-Gärtner theorem. Both approaches rely on results in \cite{b-b},
which in particular yields the asymptotics of the corresponding logarithmic
moment generating functions (i.e. free energies). Technically, one
advantage of using the Ellis-Gärtner theorem is that it avoids invoking
the variational results on the Monge-Ampère equation in \cite{bbgz}.
Then the proofs of the LDP:s in the non-compact and $\beta-$deformed
settings is given, by reducing to the previous setting. The paper
is concluded by exploring relations to bosonization in quantum field
theory in section \ref{sec:Relation-to-bosonization}.

\subsection*{Acknowledgments}

Thanks to Sébastien Boucksom, Vincent Guedj and Ahmed Zeriahi for
the stimulation coming from the collaboration \cite{bbgz}. Also thanks
to Sébastien Boucksom and Norm Levenberg for comments on the first
version of the present paper and to Steve Zelditch for urging me to
develop the relations to bosonization. A (partial) large deviation
principle with an equivalent rate functional in the case when $X$
is the complex projective space was also announced by Tom Bloom \cite{bl}
and I would also like to thank him for some related discussions.

\subsection{\label{sub:Notation-and-general}Notation%
\footnote{general references for this section are the books \cite{gr-ha,de4}.
See also \cite{b-v-} for the Riemann surface case.%
}}

Let $L\rightarrow X$ be a holomorphic line bundle over a compact
complex manifold $X.$

\subsubsection{Metrics on $L$}

We will fix, once and for all, a continuous Hermitian metric $\left\Vert \cdot\right\Vert $
on $L.$ Its curvature current times the normalization factor $\frac{i}{2\pi}$
will be denoted by $\omega.$ The normalization is made so that $[\omega]$
defines an\emph{ integer} cohomology class, i.e. $[\omega]\in H^{2}(X,\Z).$
The local description of $\left\Vert \cdot\right\Vert $ is as follows:
let $s$ be a trivializing local holomorphic section of $L,$ i.e.
$s$ is non-vanishing an a given open set $U$ in $X.$ Then we define
the local\emph{ weight }$\Phi$ of the metric $\left\Vert \cdot\right\Vert $
by the relation \[
\left\Vert s\right\Vert ^{2}=e^{-\Phi}\]
The (normalized) curvature current $\omega$ may now by defined by
the following expression: \[
\omega=\frac{i}{2\pi}\partial\overline{\partial}\Phi:=dd^{c}\Phi,\]
 (where we, as usual, have introduced the real operator $d^{c}:=i(-\partial+\overline{\partial})/4\pi$
to absorb the factor $\frac{i}{2\pi}).$ The point is that, even though
the function $\phi$ is merely locally well-defined the form $\omega$
is globally well-defined (as any two local weights differ by $\log|g|^{2}$
for $g$ a non-vanishing holomorphic function). The current $\omega$
is said to be\emph{ positive} if the weight $\Phi$ is \emph{plurisubharmonic
(psh). }If $\Phi$ is smooth this simply means that the\emph{ }Hermitian
matrix $\omega_{ij}=(\frac{\partial^{2}\Phi}{\partial z_{i}\partial\bar{z_{j}}})$
is positive definite (i.e. $\omega$ is a Kähler form) and in general
it means that, locally, $\Phi$ can be written as a decreasing limit
of such smooth functions. Finally, we recall that from the point of
view of gauge theory the (non-normalized) curvature form $\partial\overline{\partial}\Phi$
is the curvature form $F_{A}$ of the Chern connection on the complex
line bundle $L,$ i.e. the unique connection $A$ on $L$ which is
compatible with its given holomorphic structure and Hermitian metric
$\left\Vert \cdot\right\Vert .$

\subsubsection{Holomorphic sections of $L$}

We will denote by $H^{0}(X,L)$ the space of all global holomorphic
sections of $L.$ In a local trivialization as above any element $\Psi$
in $H^{0}(X,L)$ may be represented by a local holomorphic function
$f,$ i.e. \[
\Psi=fs\]
 The squared point-wise norm $\left\Vert \Psi\right\Vert ^{2}(x)$
of $\Psi,$ which is a globally well-defined function on $X,$ may
hence be locally written as \[
\left\Vert \Psi\right\Vert ^{2}(x)=(|f|^{2}e^{-\Phi})(x)\]
It will be convenient to take the curvature current $\omega$ as our
geometric data associated to the line bundle $L.$ Strictly speaking,
it only determines the metric $\left\Vert \cdot\right\Vert $ up to
a multiplicative constant but all the geometric and probabilistic
constructions that we will make are independent of the constant.

\subsubsection{Metrics and weights vs $\omega-$ psh functions}

Having fixed a continuous Hermitian metric $\left\Vert \cdot\right\Vert $
on $L$ with (local) weight $\Phi_{0}$ any other metric may be written
as \[
\left\Vert \cdot\right\Vert _{\phi}^{2}:=e^{-\phi}\left\Vert \cdot\right\Vert ^{2}\]
for a continuous function $\phi$ on $X,$ i.e. $\phi\in C^{0}(X).$
In other words, the local weight of the metric $\left\Vert \cdot\right\Vert _{\phi}$
may be written as $\Phi=\phi+\Phi_{0}$ and hence its curvature current
may be written as \[
dd^{c}\Phi=\omega+dd^{c}\phi:=\omega_{\phi}\]
This means that we have a correspondence between the space of all
(singular) metrics on $L$ with positive curvature current and the
space $PSH(X,\omega)$ of all upper-semi continuous functions on $X$
such that $\omega_{\phi}\geq0$ in the sense of currents. Note for
example, that if $\Psi\in H^{0}(X,L)$ then $\log\left\Vert \Psi\right\Vert ^{2}\in PSH(X,\omega).$

\subsubsection{Weighted sets and measures and the induced norms on $H^{0}(X,L)$}

Given a compact subset $K$ of $X$ we will call the pair $(K,\omega)$
a weighted set (where we recall that $\omega$ is normalized curvature
form of a fixed metric $\left\Vert \cdot\right\Vert ^{2}$ on $L)$
. To any such weighted set we can associate the following $L^{\infty}-$norm
on $H^{0}(X,L):$ \[
\left\Vert \Psi\right\Vert _{L^{\infty}(K)}:=\sup_{x\in K}\left\Vert \Psi(x)\right\Vert \]
This is a non-degenerate norm if $K$ is not contained in a analytic
subvariety of $X$ and in particular if $K$ is non-pluripolar, i.e.
$K$ is not locally contained in the $-\infty$ set of a plurisubharmonic
function. We will fix a weighted compact non-pluripolar subset $K$
once and for all.

Similarly, given a probability measure $\nu$ on $X$ we will call
the pair $(\nu,\omega)$ a \emph{weighted measure}. It induces a $L^{2}-$norm
on $H^{0}(X,L)$ (which will always be non-degenerate in the present
paper): \[
\left\Vert \Psi\right\Vert _{L^{2}(K,\nu)}^{2}:=\int_{X}\left\Vert \Psi\right\Vert ^{2}\nu\]
Sometimes we will also use the notation \[
\left\Vert \Psi\right\Vert _{L^{2}(K,e^{-\phi})}^{2}:=\int_{X}\left\Vert \Psi\right\Vert _{\phi}^{2}\nu=\int_{X}\left\Vert \Psi\right\Vert ^{2}e^{-\phi}\nu\]
if $\phi\in C^{0}(X)$ and similarly for $L^{\infty}-$norms.

\subsection{Scaling by $k$}

The Hermitian line bundle $(L,\Phi)$ over $X$ induces, in a functorial
way, Hermitian line bundles over all products of $X$ (and its conjugate
$\overline{X}$) and we will usually keep the notation $\Phi$ for
the corresponding weights. When studying asymptotics we will replace
$L$ by its $k$ th tensor power, written as $kL$ in additive notation.
The induced weights on $kL$ may then be written as $k\Phi.$

\section{\label{sec:Examples}Examples}

\subsection{Multivariate polynomial ensembles}

Let $\nu$ be a measure supported on a compact subset $K$ of $\C^{n}.$
Identifying $\C^{n}$ with an affine piece of the projective space
$X:=\P^{n}$ induces a trivialization of the hyperplane line bundle
$\mathcal{O}(1)\rightarrow\P^{n}$ so that a locally bounded metric
gets identified with a weight $\Phi$ defined on $\C^{n}$ of logarithmic
growth: \begin{equation}
\Phi(z)=\Phi_{0}(z)+O(1):=\log^{+}\left|z\right|^{2}+O(1)\label{eq:log growth}\end{equation}
where $\Phi_{0}(z):=\log^{+}\left|z\right|^{2}=\log\left|z\right|^{2}$
for $|z|>1$ and $0$ otherwise (see for example the appendix in \cite{berm2}).
Fixing $\omega_{0}:=dd^{c}\log^{+}\left|z\right|^{2},$ which extends
to a form on $\P^{n}$ with locally continuous potentials,  and letting
$\phi=\Phi-\log^{+}\left|z\right|^{2}$ hence yields a bijection between
all\emph{ bounded }$\phi\in PSH(\P^{n},\omega_{0})$ and all $\Phi$
as above which are psh on $\C^{n}.$ Moreover the space $H^{0}(\P^{n},k\mathcal{O}(1))$
gets identified with the space of all polynomials $p_{k}(z)$ in $z_{1},,,.,z_{n}$
on $\C^{n}$ of total degree at most $k$ and the point-wise norm
induced by $\Phi$ is given by 

\[
\left\Vert \Psi_{k}\right\Vert ^{2}(z):=|p_{k}(z)|^{2}e^{-\Phi(z)}\]
(so that $\omega=dd^{c}\Phi$ in the previous notation). Note that
since $K$ is compact the classical unweighted theory in $\C^{n}$
may be obtained by taking $\Phi=0$ on $K$ and then extending $\Phi$
so that \ref{eq:log growth} holds.

Now the pair $(\nu,\Phi)$ defines, for any $k,$ a determinantal
point process on $K$ concretely obtained by taking $\Psi_{i}^{(k)}$
to be a base for the space of all polynomials in $z_{1},,,.,z_{n}$
on $\C^{n}$ of degree at most $k.$ If the base consists of multinomials
then the corresponding Slater determinant is known as the (multivariate)
\emph{Vandermonde determinant}: \begin{equation}
\det\Psi_{k}=\Delta^{(N_{k})}(z_{1},...,z_{N_{k}})\label{eq:vandermo}\end{equation}
The conditions in Theorem \ref{thm:intro large dev} are satisfied
if, for example, $K$ is compact domain in $\C^{n}$ with smooth boundary
or its boundary and $\nu$ is the measure defined by a volume form
on $K$ (see section \ref{sub:Bernstein-Markov-measures}).

More generally, we can replace $\C^{n}$ with an affine algebraic
subvariety $X_{0}:=\{p_{1}(z)=...p_{m}(z)=0\}$ cut out by polynomials
$p_{i}$ on $\C^{n}$ and $\nu$ with a measure supported on a compact
subset of $X_{0}.$ Then we let $X$ be the associated projective
variety obtained by taking the Zariski-closure of $X_{0}$ in $\P^{n}$
and let $L$ be the restriction $\mathcal{O}_{X}(1)$ (strictly speaking
we have to assume that $X$ is non-singular, but otherwise one could
pass to a suitable resolution of $X$). The Slater determinant is
then defined in terms of a given base in $H^{0}(X,\mathcal{O}_{X}(1)),$
i.e. the vector space spanned by the restriction to $X$ of all polynomials
of degree at most $k$ in $\C^{n}.$ Again the conditions in Theorem
\ref{thm:intro large dev} are satisfied if $K$ is a bounded domain
in $X$ or its boundary and $\nu$ is a measure as above.

\subsection{\label{sub:Spherical-and-trigonmetric}Real examples}

It is interesting to apply the previous setup to a completely {}``real''
setting. For example, we can take the measure $\nu$ to be supported
on a compact subset $K$ of $\R^{n}.$ Embedding $\R^{n}$ in $\C^{n}$
and taking a base of polynomials defined over $\R,$ i.e. with real
coefficients, then induces a determinantal point-process on $K$ to
which Theorem \ref{thm:intro large dev} applies if, for example,
$K$ is a smooth domain in $\R^{n}$ and $\nu$ is taken as the usual
Euclidean (Lebesgue) measure on $K$ (as follows from results in \cite{bl0}).
When $n=1$ the corresponding ensemble may be realized by random Hermitian
matrices with eigenvalues conditioned to lie in $K$ (a finite union
of intervals) \cite{dei2}.

Similarly, if the polynomials $p_{1},...p_{m}$ defining the affine
algebraic variety $X_{0}$ above have real coefficients and the corresponding
real algebraic variety $K:=X_{0}\cap\R^{n}$ is compact then we can
take $\nu$ to be the measure on $X$ defined by any given volume
form on $K\subset X.$ The proofs of the Bernstein-Markov properties
in this general setting will appear elsewhere. The following two particular
cases have been extensively studied in the literature, in particular
in the context of approximation theory, and as we will explain the
Bernstein-Markov properties can be proved directly in these cases.

\subsubsection{Spherical polynomials }

Let $K=S^{n}$ be the unit $n-$sphere in $\R^{n+1}$ and $\nu$ the
usual $O(n,\R)-$invariant probability measure on $S^{n}.$ We can
then let the algebraic variety $X_{0}$ above be the complex quadric
defined by $\{p(z)=z_{1}^{2}+...z_{n+1}^{2}-1=0\}$ so that its real
points are precisely $S^{n}.$ Then $H^{0}(X,\mathcal{O}_{X}(1))$
with the Hermitian product induced by $(\nu,0)$ gets identified with
the complexification of the space $H_{k}(S^{n})$ of all spherical
polynomials on the $n-$sphere $S^{n}$ of degree at most $k$ equipped
with its usual $O(n,\R)-$invariant scalar product. The corresponding
Slater determinant naturally appears in numerical problems as the
determinant of the interpolation matrix (see for example \cite{sl-w}
and references therein). From the physics point of view the process
represents the ground state of a gas of free spin-polarized fermions
on $S^{n}$ (which can be seen as a compact version of the $\R^{n}-$case
studied in \cite{t-s-z,t-s-z2} in connection to sphere packings etc).

It may be illuminating to give the following direct proof of the Bernstein-Markov
property (section \ref{sub:Bernstein-Markov-measures}) in this case.
From the $O(n,\R)-$invariance it follows immediately that $\E(\delta)=\nu$
and hence it, trivially, has sub-exponential decay, i.e. $(\nu,0)$
has the BM-property. But then it follows from general principles that
it in fact has the BM-property wrt \emph{any} continuous function
on $K:=S^{n}.$ This was shown by Bloom in the case when $K\subset\R^{n}$
\cite{bl0} (Thm 3.2), but the same arguments work in the present
case. Indeed, writing \[
|p_{k}(z)|^{2}e^{-k\phi}=|p_{k}(z)(e^{-\phi/2})^{k}|^{2}\]
for a continuous function $\phi$ on $S^{n}$ we may extend $-\phi$
continuously to all of $\R^{n}$ and then, using the Stone-Weierstrass
theorem approximate it uniformly by polynomials on $\R^{n}.$ Truncating
the Taylor expansion of $e^{-\phi}$ then allows us to approximate
$e^{-\phi}$ with polynomials $p_{\epsilon}$ such that \[
1-2\epsilon\leq p_{\epsilon}/e^{-\phi}\leq1+2\epsilon\]
Applying the BM-property for $(\nu,0)$ to the polynomial $p_{k}(z)p_{\epsilon}(z)$
and rescaling then proves that $(\nu,e^{-\phi})$ also has the BM-property.

\subsubsection{Trigonometric polynomials}

We let $K:=T^{n}$ be the unit $n-$torus in $\R_{x,y}^{2n}$ and
set $p_{i}(z,w)=z_{i}^{2}+w_{i}^{2}-1$ for $i=1,...n$ in $\C_{z,w}^{2n}$
so that $K\subset X\subset\P^{2n}.$ Then $H^{0}(X,\mathcal{O}_{X}(1))$
with the Hermitian product induced by $(\nu,0)$ gets identified with
the complexification of the space $H_{k}(T^{n})$ of all trigonometric
polynomials on $T^{n}$ of total degree at most $k$ (i.e. the corresponding
frequencies lie in $k$ times the unit simplex) equipped with its
usual $O(2,\R)^{\otimes n}-$invariant scalar product. Similarly,
replacing $\P^{2n}$ with $(\P^{2})$ gives trigonometric polynomials
with frequencies in $[0,k]^{n}.$

\subsection{\label{sub:The-non-compact-setting}The case when $K$ is a closed
non-compact subset of $\C^{n}$ }

There is a slight non-compact variant of the previous setting when
$K$ is assumed to be a merely \emph{closed} subset of $\C^{n}$ ,
but where the continuous weight $\Phi$ on $\C^{n}$ has super logarithmic
growth: \begin{equation}
\Phi(z)\geq(1+\epsilon)\ln\left|z\right|^{2},\,\,\textrm{when\,}\left|z\right|>>1\label{eq:ass on growth of phi}\end{equation}
 In particular $\Phi$ is \emph{not }the restriction of a locally
bounded metric on $\mathcal{O}(1)\rightarrow\P^{n}.$ In the case
when $n=1$ this is the setting of weighted potential theory considered
in the book \cite{s-t} (see also the appendix in \cite{s-t} for
the case when $n>1)$. In particular, we may then take $K$ as all
of $\C^{n}$ (or $\R^{n})$ and $\nu=d\lambda$ as the corresponding
Lebesgue measure. For example, in the case of $n=1$ the corresponding
ensembles may be realized by random normal (or Hermitian) random matrices
whose eigenvalues are subject to the {}``confining potential'' $\Phi$
\cite{za,dei2}. As shown in section \ref{sub:The-non-compact-setting}
the LDP in Theorem \ref{thm:intro large dev} is still valid in this
non-compact setting.

\section{The Monge-Ampere operator, energy and rate functional\label{sec:The-Monge-Ampere-operator,}s}

In this section we will recall the global complex pluripotential theory
that is needed to define and study the rate functional for the large
deviation principle. Almost all the material on the complex Monge-Ampere
equation is contained in \cite{bbgz,begz} (apart from the results
in section \ref{sub:Further-properties-of}).

\subsection{\label{sub:The-Monge-Amp=0000E8re-operator}The Monge-Ampère operator
and the functional $\mathcal{E_{\omega}}(u)$}

Let $(X,\omega)$ be a compact complex manifold equipped with a $(1,1)-$form
$\omega$ with continuous local potentials. Let us start by recalling
the definition of the Monge-Ampère measure $MA_{\omega}(\phi)$ on\emph{
smooth} functions $\phi$ (when $\omega$ has smooth potentials).
It is defined by \[
MA_{\omega}(\phi):=\frac{(\omega+dd^{c}\phi)^{n}}{Vn!}=:\frac{(\omega_{\phi})^{n}}{Vn!}\]
 where the normalization constant $V$ ensures that $MA_{\omega}(\phi)$
has total unit charge. For simplicity we will omit the subscript $\omega$
in the definition of $MA_{\omega}$ 

When $\phi$ is a \emph{Kähler potential, }i.e. $\phi$ is smooth
and $\omega_{\phi}>0$ the form $MA(\phi)$ is hence a volume form
giving unit volume to $X$ (by Stokes theorem).

The Monge-Ampère $MA$ operator may be naturally identified with a
one-form on the vector space $C^{\infty}(X)$ by letting \[
\left\langle MA_{|\phi},v\right\rangle :=\int_{X}MA(\phi)v\]
for $\phi,v\in C^{\infty}(X).$ As observed by Mabuchi, in the context
of Kähler-Einstein geometry, the one-form $MA$ is closed and hence
it has a primitive $\mathcal{E}$ (defined up to an additive constant)
on the space $C^{\infty}(X),$ i.e. \begin{equation}
d\mathcal{E}_{|\phi}=MA(\phi)\label{eq:def prop of energy on psh}\end{equation}
We fix the additive constant by requiring $\mathcal{E}(0)=0.$ Sometimes
we will use a sub-script $\omega$ to indicate the dependence of $\mathcal{E}$
on $\omega.$ Integrating $\mathcal{E}_{\omega}$ along line segments
one arrives at the following well-known formula \begin{equation}
\mathcal{E}_{\omega}(\phi):=\frac{1}{(n+1)!V}\sum_{j=0}^{n}\int_{X}\omega_{\phi}^{j}\wedge(\omega)^{n-j}\label{eq:bi-energy}\end{equation}
\footnote{Hence, $\mathcal{E}_{\omega}(\phi)=\tilde{ch}(he_{0}^{-u},h_{0})$
where $h_{0}$ is the metric whose curvature form is $\omega$ and
$\tilde{ch}(h_{1},h_{0})$ is (up to normalization) the secondary
Bott-Chern class attached to the first Chern class of $L.$\cite{so} %
} Conversely, one can simply take this latter formula as the definition
of $\mathcal{E}_{\omega}$ and observe that the following proposition
holds (compare \cite{b-b,begz,bbgz} for a more general singular setting). 
\begin{prop}
\label{pro:of energy-1}The following holds
\begin{itemize}
\item The differential of the functional $\mathcal{E}_{\omega}$ at a smooth
function $\phi$ is represented by the measure $MA(\phi),$ i.e. \begin{equation}
\frac{d}{dt}_{t=0}(\mathcal{E}_{\omega}(\phi+tv))=\int_{X}MA(\phi)v\label{eq:diff}\end{equation}

\item $\mathcal{E}_{\omega}$ is increasing on the space of all smooth $\omega-$psh
functions
\item $\mathcal{E}_{\omega}$ is concave on the space of all smooth smooth
$\omega-$psh functions (when $n=1$ it is concave on all of $C^{\infty}(X)).$
\end{itemize}
\end{prop}
Note that the first point implies the second one, since the differential
of $\mathcal{E}_{\omega}$ is represented by a (positive) measure.

\subsubsection{The singular setting and the space $\mathcal{E}^{1}(X,\omega)$}

The subspace $\mathcal{E}^{1}(X,\omega)$ of $PSH(X,\omega)$ consisting
of all $\omega-$psh functions of \emph{finite energy }is defined
as follows (generalizing the classical Dirichlet spaces on Riemann
surfaces). First we extend the functional $\mathcal{E}_{\omega}$
(formula \ref{eq:bi-energy}) to all $\omega-$psh functions by demanding
that it still be increasing, i.e. we define \[
\mathcal{E}_{\omega}(\phi):=\inf_{\psi\geq\mathbf{\phi}}\mathcal{E}_{\omega}(\psi)\in[-\infty,\infty[\]
 where $\psi$ ranges over all smooth $\omega-$psh functions such
that $\psi\geq\phi.$ Next, we let \[
\mathcal{E}^{1}(X,\omega):=\{\phi\in PSH(X,\omega):\,\mathcal{E}_{\omega}(\phi)>-\infty\},\]
 which is a convex subspace, since $\mathcal{E_{\omega}}$ is concave.
As a consequence of the monotonicity of $\mathcal{E}_{\omega}(u)$
and Bedford-Taylor's fundamental local continuity result for mixed
Monge-Ampère operators one obtains the following proposition (cf.
\cite{begz}, Prop 2.10; note that $\mathcal{E}_{\omega}=-E_{\chi}$
for $\chi(t)=t$ in the notation in op. cit.) 
\begin{prop}
\label{pro:energy is ups}The functional $\mathcal{E}_{\omega}(u)$
is upper semi-continuous on $PSH(X,\omega),$ concave and non-decreasing.
Moreover, it is continuous wrt decreasing sequences in $PSH(X,\omega).$
\end{prop}
For any $u\in\mathcal{E}^{1}(X,\omega)$ the (non-pluripolar) Monge-Ampère
measure $MA(u)$ is well-defined \cite{begz} and does not charge
any pluripolar sets. We collect the continuity properties that we
will use in the following \cite{begz} 
\begin{prop}
\label{pro:non-pluripol ma}Let $(u^{(i)})\subset\mathcal{E}^{1}(X,\omega)$
be a sequence decreasing to $u\in\mathcal{E}^{1}(X,\omega).$ Then,
as $i\rightarrow\infty,$ \[
MA(u_{i})\rightarrow MA(u),\,\,\,\,\,\,\,\,\, u_{i}MA(u_{i})\rightarrow uMA(u)\]
in the weak topology of measures and $\mathcal{E}_{\omega}(u_{j})\rightarrow\mathcal{E}_{\omega}(u).$ 
\end{prop}

\subsection{The pluricomplex energy $E_{\omega}(\mu)$ and potentials}

Following \cite{begz} we define the (pluricomplex) energy by \begin{equation}
E_{\omega}(\mu):=\sup_{\phi\in PSH(X,\omega)}\mathcal{E}_{\omega}(\phi)-\left\langle \phi,\mu\right\rangle \label{eq:def of e as sup}\end{equation}
if $\mu\in\mathcal{M}_{1}(X)$ (sometimes we will omit the subscript
$\omega$ and simple write $E_{\omega}=E).$ We will denote the subspace
of all finite energy probability measures by \[
E_{1}(X):=\{\mu:\, E_{\omega}(\mu)<\infty\}\]
(which only depends on the\emph{ class} $[\omega]$ and not on the
representative $\omega).$

By Propositions \ref{pro:non-pluripol ma} and Demailly's approximation
theorem it is enough to take the sup over all Kähler potentials. But
one point of working with less regular functions is that the sup can
be attained. Indeed, as recalled in the following theorem \begin{equation}
E_{\omega}(\mu):=\mathcal{E}_{\omega}(\phi_{\mu})-\left\langle \phi_{\mu},\mu\right\rangle \label{eq:energy in terms of pot}\end{equation}
for a unique potential $\phi_{\mu}\in\mathcal{E}^{1}(X,\omega)/\R$
of the measure $\mu$ if $E_{\omega}(\mu)<\infty$ where \begin{equation}
MA(\phi_{\mu})=\mu\label{eq:potential of meas}\end{equation}

\begin{thm}
\label{thm:var sol of ma}\cite{bbgz} The following is equivalent
for a probability measure $\mu$ on $X:$ 
\begin{itemize}
\item $E_{\omega}(\mu)<\infty$
\item \textup{$\left\langle \phi,\mu\right\rangle <\infty$ for all $\phi\in\mathcal{E}^{1}(X,\omega)$}
\item $\mu$ has a potential $\phi_{\mu}\in\mathcal{E}(X,\omega$), i.e.
equation \ref{eq:potential of meas} holds
\end{itemize}
Moreover, $\phi_{\mu}$ is uniquely determined mod $\R,$ i.e. up
to an additive constant and can be characterized as the function maximizing
the functional whose sup defines $E_{\omega}(\mu)$ (formula \ref{eq:def of e as sup}).
Even more generally: if $\phi_{j}$ is an asymptotically maximizing
sequence sequence (normalized so that $\sup_{X}\phi_{j}=0),$ i.e.
\[
\liminf_{j\rightarrow\infty}\mathcal{E}(\phi_{j})-\left\langle \phi_{j},\mu\right\rangle =E_{\omega}(\mu)\]
then $\phi_{j}\rightarrow\phi_{\mu}$ in $L^{1}(X,\mu)$ and $\mathcal{E}(\phi_{j})\rightarrow\mathcal{E}(\phi_{j}).$

\end{thm}
\begin{rem}
In the proof of the large deviation principle (LDP) in section \ref{sec:Large-deviations}
we will only use the existence of a potential $\phi_{\mu}$ for a
given measure $\mu$ of finite energy (and not the uniqueness). As
for the maximization property of $\phi_{\mu}$ it will follow from
the proof of the LDP, but it is also a simple consequence of the concavity
of the functional $\mathcal{E}_{\omega}$ on the space $\mathcal{E}^{1}(X,\omega).$
\end{rem}
The previous theorem was proved in \cite{bbgz} using the variational
approach in the more general setting of a big class $[\omega]$ -
one crucial ingredient in the proof is the differentiability theorem
\ref{thm:diff thm} below. In the case when $\mu$ is a volume form
Yau's seminal theorem \cite{y} furnishes a \emph{smooth }potential
$\phi_{\mu},$ i.e a Kähler potential (using the continuity method
for PDEs and delicate a priori estimates).

\subsection{\label{sub:The-psh-projection}The psh projection $P$ and the equilibrium
measure}

Given a compact non-pluripolar set $K$ and a (possibly singular)
function $\phi$ one defines the $\omega-$psh function $P_{(K,\omega)}\phi)(x)$
as the following regularized upper envelope: \begin{equation}
(P_{(K,\omega)}\phi)(x):=\left(\sup\left\{ \psi(x):\,\psi\in PSH(X,\omega),\,\psi\leq\phi\,\,\textrm{on\ensuremath{\, K}}\right\} \right)^{*}\label{eq:def of proj as reg env}\end{equation}
 where the star denotes upper semi-continuous regularization (we will
often omit the sub-script $\omega)$. If $\phi$ is continuous, then
$P_{(K,\omega)}\phi$ is locally bounded precisely when $K$ is non-pluripolar
(see section \ref{sub:Bernstein-Markov-measures}). Now the pluripotential
equilibrium measure of a weighted non-pluripolar set $(K,\omega)$
may be defined as the following measure \[
\mu_{eq}:=MA(P_{(K,\omega)}0)\]
supported on $X.$ This is the global version of the original definition
given by Siciak in the context of approximation theory in $\C^{n}$
(see \cite{g-z} and references there in). An alternative \emph{variational}
characterization of the equilibrium measure was given very recently
in \cite{bbgz}, which will play a prominent rule in this paper (see
below).

When $K=X$ we will simply write \[
P_{(X,\omega)}=P_{\omega}(=P)\]
If $\phi$ is continuous then so is $P_{\omega}\phi,$ even without
using the upper semi-continuous regularization. Indeed, the lower
semi-continuity of $P_{\omega}\phi$ follows from Demailly's approximation
result which allows us to write $P_{\omega}\phi$ as an upper envelope
of continuous functions and the upper semi-continuity is obtained
by noting that $P_{\omega}\phi$ is a candidate for the sup in its
definition \cite{g-z,b-b}. One of the main results in \cite{b-b}
is the following
\begin{thm}
\label{thm:diff thm}(B.-Boucksom \cite{b-b}) Let $K$ be a compact
non-pluripolar subset of $X.$ Then the functional $\mathcal{E_{\omega}}\circ P_{(K,\omega)}$
is concave and Gateaux differentiable on $C^{0}(X).$ More precisely,
\[
d(\mathcal{E_{\omega}}\circ P_{(K,\omega)})_{|\phi}=MA(\mathcal{E_{\omega}}(P_{(K,\omega)}\phi))\]

\end{thm}
It should be emphasized that the differentiability result above is
in a sense very surprising (even when $\phi$ is smooth and $K=X)$
. Indeed, the projection operator $P_{(K,\omega)}$ is certainly not
differentiable. Moreover, the functional $\mathcal{E_{\omega}}\circ P_{(K,\omega)}$
is in general not\emph{ two} times differentiable. From a statistical
mechanical point of view the one time differentiability corresponds
to an absence of a first order phase transition (see \cite{berm3}).
An important ingredient in the proof of the previous theorem is the
orthogonality relation (itself a consequence of the maximum principle
for $MA)$ \begin{equation}
\left\langle MA(P_{(K,\omega)}\phi)),\phi-P_{(K,\omega)}\phi\right\rangle =0\label{eq:og relatgion}\end{equation}
saying that $P_{(K,\omega)}\phi=\phi$ a.e. wrt the measure $MA(P_{(K,\omega)}\phi)$

\subsection{\label{sub:Further-properties-of}Further properties of the energy
$E_{\omega}(\mu)$}
\begin{prop}
\label{pro:prop of s}The following properties of the energy $E(\mu)$
(formula \ref{eq:def of e as sup}) hold:
\begin{itemize}
\item Assume that $E(\mu)<\infty.$ Then the sup defining the energy $E(\mu)$
may be taken over the subset of all \emph{continuous} $\omega-$psh
functions. More generally, if $\mu$ is supported on a compact set
$K$ in $X,$ then\begin{equation}
E(\mu):=\sup_{\phi}\mathcal{E}(P_{K}\phi)-\int_{X}\phi\mu\label{eq:s in terms of cont}\end{equation}
 where $\phi$ ranges over either all continuous $\omega-$psh functions
on $X$ or over all of $C^{0}(X).$ 
\item The functional $E$ is lower semi-continuous (lsc) on $\mathcal{P}(K)$
if $K$ is compact in $X$
\end{itemize}
\end{prop}
\begin{proof}
Given the set $K$ and $\phi\in C^{0}(X)\cap PSH(X,\omega)$ we note
that. by definition, $P_{K}\phi\geq\phi$ on $X$ (and $P_{K}\phi=\phi$
on $K).$ Hence, $\mathcal{E}(P_{K}\phi)-\int_{X}\phi\mu\geq\mathcal{E}(\phi)-\int_{X}\phi\mu$
and since $P_{K}\phi$ is a candidate for the sup defining $E(\mu)$
this proves the statement when $\phi$ ranges over continuous $\omega-$psh
functions. Next, if $\phi$ is merely continuous then we decompose
\[
\mathcal{E}(P_{K}\phi)-\int_{X}\phi\mu=\left(\mathcal{E}(P_{K}\phi)-\int_{X}P_{K}\phi\mu\right)+\int_{X}(P_{K}\phi-\phi)\mu\]
As is well-known $P_{K}\phi\leq\phi$ quasi-everywhere, i.e. away
from a pluripolar set. But since $E(\mu)<\infty$ the measure $\mu$
does not charge pluripolar sets \cite{bbgz} and hence setting $\psi:=P_{K}\phi$
gives $\mathcal{E}(\psi)-\int_{X}\psi\mu\leq E(\mu).$ Finally, writing
$\psi$ as a decreasing limit of elements $\psi_{j}\in C^{0}(X)\cap PSH(X,\omega)$
and using the previous case for $\phi=\psi_{j}$ finishes the proof
of the first point.

As for the lower semi-continuity of $E$ it follows immediately from
the fact that $E$ is defined as a sup of continuous functionals.
\end{proof}
We will also have use for the following approximation lemma of independent
interest
\begin{lem}
\label{lem:appr}Assume that $\mu$ is a probability measure supported
on a compact set $K$ such that $E(\mu)$ is finite. Let $\phi_{\mu}$
be a potential of $\mu$ and take a sequence $\phi_{j}$ in $C^{0}(X)\cap PSH(X,\omega)$
such that $\phi_{j}$ decreases to $\phi.$ Then $\mu_{j}:=MA(P_{K}\phi_{j})\rightarrow\mu$
and $E(\mu_{j})\rightarrow\mu.$\end{lem}
\begin{proof}
Since by definition $P_{K}\phi_{j}\geq\phi_{j}\geq\phi_{\mu}$ and
$P_{K}\phi_{j}=\phi_{j}$ $\mu-$a.e. we have \[
\mathcal{E}(P_{K}\phi_{j})-\int_{X}P_{K}\phi_{j}\mu\geq\mathcal{E}(\phi)-\int_{X}\phi_{j}\mu\rightarrow\mathcal{E}(\phi_{\mu})-\int_{X}\phi_{\mu}\mu(=E(\mu))\]
using the monotone convergence theorem in the last step. In other
words, the sequence $P_{K}\phi_{j}$ is asymptotically maximizing
for the functional whose unique maximizer is $\mu$ and hence it follows
from Theorem 4.7 in \cite{bbgz} that $P_{K}\phi_{j}\rightarrow\phi_{\mu}+c$
{}``in energy'' for $c\in\R.$ But integrating against $\mu$ and
using that $P_{K}\phi_{j}=P_{K}\phi$ on the set $K$ (where $\mu$
is supported) forces $c=0.$ For the precise definition of {}``convergence
in energy'' see \cite{bbgz}, but the main point is that it implies
convergence of all relevant mixed Monge-Ampère measures (see Lemma
3.2 in \cite{bbgz}) and in particular $MA(P_{K}\phi_{j})\rightarrow\mu$
and $E(\mu_{j})\rightarrow E(\mu).$\end{proof}
\begin{rem}
The proof of the previous lemma also shows that $P_{K}\phi_{\mu}=\phi_{\mu}$
for any measure $\mu$ of finite energy supported on a compact set
$K$ (when $\phi_{\mu}$ is continuous this is an immediate consequence
of the standard domination principle). Indeed, by definition $P_{K}\phi_{j}\geq P_{K}\phi_{\mu}\geq\phi_{\mu}$
and by the proof above $P_{K}\phi_{j}\rightarrow\phi_{\mu}.$ This
property was used without any explicit proof in the previous version
of the paper (thanks to Norm Levenberg for pointing this out) and
a different proof by Dinew was supplied in \cite{b-l-1}. 
\end{rem}
The following proposition gives an explicit expression for the energy
of a measure in terms of the potential $\phi_{\mu}.$ The proof is
obtained from a straight-forward calculation and is hence omitted. 
\begin{prop}
The energy $E_{\omega}(\mu)$ of a probability measure $\mu$ with
$E_{\omega}(\mu)<\infty$ may be written as \[
E_{\omega}(\mu)=\frac{1}{V}\sum_{j=0}^{n-1}\frac{1}{j+2}\int d\phi_{\mu}\wedge d^{c}\phi_{\mu}\wedge(dd^{c}\phi_{\mu})_{j}\wedge\omega_{n-1-j}\]
where we have used the notation $\eta_{p}:=\eta^{p}/p!$
\end{prop}
Even though we will not use the previous proposition in the proofs
it will appear in the discussion in section \ref{sec:Relation-to-bosonization}.
Note when $n=1$ $E_{\omega}$ is hence a multiple of the classical
Dirichlet energy and may also be expressed as \begin{equation}
E_{\omega}(\mu)=-\frac{1}{2}\int G_{\omega}(x,y)\mu(x)\otimes\mu(y)\label{eq:energy as green}\end{equation}
(where we have assumed $V=1$ for simplicity), where $G_{\omega}(x,y)$
is the Green function defined by $d_{x}d_{x}^{c}G_{\omega}(x,y)=\delta_{y}(x)-\omega(x)$
and the normalization condition $\int G(x,y)\omega(y)=0.$

\subsection{The rate functional, electrostatic capacity and the analytic torsion}

Given a (weighted) non-pluripolar compact set $K$ we define the \emph{rate
functional} as the normalized energy functional: \begin{equation}
H_{(K,\omega)}(\mu):=E_{\omega}(\mu)-C(K,\omega)\label{eq:rate function as e minus cap}\end{equation}
where $C(K,\omega)$ is the following constant \begin{equation}
C(K,\omega):=\inf_{\mu\preceq K}E_{\omega}(\mu)\label{eq:capacity as inf}\end{equation}
The constant $e^{-\frac{n}{n+1}C(K,\omega)}$ was called the \emph{pluricomplex}
\emph{electrostatic capacity} in \cite{bbgz}; it generalizes the
logarithmic capacity in $\mbox{\C}$ and Leja's transfinite diameter
in $\C^{n}.$ Moreover, as shown in \cite{bbgz} \begin{equation}
C(K,\omega):=\mathcal{E}_{\omega}(P_{(K,\omega)}0)<\infty\label{eq:cap as enrgy with proj}\end{equation}
 and hence the rate functional $H_{(K,\omega)},$ defined above, may
also be expressed as \begin{equation}
H_{(K,\omega)}=E_{\omega}(\mu)-\mathcal{E}_{\omega}(P_{(K,\omega)}0)\label{eq:rate functional as energies}\end{equation}
It is in this latter form that the rate functional will appear in
the proof of the large deviation principle (LDP) and as a byproduct
of the LDP we will then obtain the formula \ref{eq:capacity as inf}. 
\begin{prop}
\label{pro:rate}Let $K$ be a non-pluripolar compact subset of $X.$
The functional $H_{(K,\omega)}:\,\mathcal{P}(K)\rightarrow[0,\infty]$
is a \emph{good rate functional}, i.e. it is lower semi-continuous
and proper. It has a unique minimizer which coincides with $\mu_{eq},$
the equilibrium measure of $(K,\omega),$ defined in section \ref{sub:The-psh-projection}. \end{prop}
\begin{proof}
First observe that, clearly, the functional $H_{(K,\omega)}$ is lsc
iff the functional $E$ is lsc, which holds by Proposition \ref{pro:prop of s}.
To prove that $H_{(K,\omega)}$ is proper we must prove that the sublevel
sets $\{H_{(K,\omega)}\leq C\}$ are compact for any constant $C.$
Since $I_{I_{(K,\omega)}}$ is lsc these sets are closed in $\mathcal{P}(K).$
But by weak compactness of $\mathcal{P}(K)$ any closed set is compact.
As for the uniqueness it follows from the differentiability theorem
\ref{thm:diff thm} combined with standard convexity arguments (see
\cite{bbgz}) - alternatively it will follow from the LDP in theorem
\ref{thm:intro large dev}.
\end{proof}

\subsubsection{Analytic torsion}

In the general non-Kähler case it is shown in \cite{berm4} that the
analytic torsions in the lhs in formula \ref{eq:conv of anal tor intro}
(in the introduction) converge towards \[
Vn!(\mathcal{E_{\omega}}(P_{\omega}0)-\mathcal{E_{\omega}}(0))=Vn!(\mathcal{E_{\omega}}(P_{\omega}0)\]
(where the factors $Vn!$ come from the conventions for the functional
$\mathcal{E}$ in the present paper). Hence the convergence in \ref{eq:conv of anal tor intro}
follows from the identities \ref{eq:capacity as inf} and \ref{eq:cap as enrgy with proj}
(with $K=X).$

\subsection{\label{sub:Bernstein-Markov-measures}Bernstein-Markov measures}

Following \cite{b-b}, we will say that a measure $\nu$ has the \emph{Bernstein-Markov
property wrt the weighted set $(K,\omega)$} if, given any positive
number $\epsilon,$ there exist $C_{\epsilon}$ such that\begin{equation}
\sup_{x\in K}\left\Vert s_{k}\right\Vert ^{2}(x)\leq C_{\epsilon}e^{k\epsilon}\int_{X}\left\Vert s_{k}\right\Vert ^{2}d\nu\label{eq:bm}\end{equation}
 for any element $s_{k}$ of $H^{0}(X,kL),$ where the norms are taken
wrt the metric whose curvature current is $\omega$ (in particular,
if $K$ is non-pluripolar, then any such measure $\nu$ defines a
non-degenerate $L^{2}-$norm on the spaces $H^{0}(X,kL)).$ Moreover,
we will say that the measure $\nu$ has the\emph{ Bernstein-Markov
property wrt the set $K$} if \ref{eq:bm strong} holds for \emph{any}
$\omega$ realized as the curvature current of a continuous metric
on $L\rightarrow X.$ Equivalently, \begin{equation}
\sup_{x\in K}(\left\Vert s_{k}\right\Vert ^{2}e^{-\phi}(x))\leq C_{\epsilon}e^{k\epsilon}\int_{X}\left\Vert s_{k}\right\Vert ^{2}e^{-\phi}d\nu\label{eq:bm with phi}\end{equation}
for any given $\phi\in C^{0}(K)$ (where, of course, the constant
$C_{\epsilon}$ depends on $\phi).$ We will also say that $\nu$
has the \emph{strong Bernstein-Markov property wrt the weighted set
$(K,\omega)$ }if for any $p>0$ there is an $\epsilon$ and $C_{\epsilon}>0$
such that for all $\psi\in PSH(X,\omega)$\emph{ }\begin{equation}
\sup_{x\in K}e^{p\psi}\leq C_{\epsilon,p}e^{p\epsilon}\int_{X}e^{p\psi}d\nu\label{eq:bm strong}\end{equation}
Similarly, $\nu$ has the \emph{strong Bernstein-Markov property wrt
the set $K$ }if the property holds wrt $(K,\omega)$ for any $\omega.$
The strong BM-property implies the BM-property, as follows immediately
by setting $\psi:=\frac{1}{k}\log\left\Vert s_{k}\right\Vert ^{2}$
and $p=k.$ 
\begin{rem}
The previous terminilogy is non-standard and may appear somewhat confusing
when compared with the terminology of Bernstein-Markov measures in
the particular case when $K$ is a compact subset of $\C^{n}$ \cite{bl0}.
In this latter case $\nu$ is usually said to have the BM-propety
wrt a compact set $K$ in $\C^{n}$ if $\nu$ has the BM-property
wrt the weighted set $(K,0)$ (i.e. $\omega=0$ on $K)$ in our terminology.
But for a general complex manifolds $X$ there is no canonical choice
of form $\omega$ for a given set $K$ which is the reason for the
terminology used here.
\end{rem}

\section{Proof of the Large Deviation\label{sec:Large-deviations} Principles}

\subsection{Definition of the determinantal probability measure}

To any given weighted measure $(\nu,\omega)$ we associate a sequence
of probability measures $\mu^{(N_{k})}$ on $X^{N_{k}}$ defined as
follows. First we set $k=1$ and recall that $N$ denotes the dimension
of the vector space $H^{0}(X,L).$ Hence, the top exterior power $\Lambda^{N}H^{0}(X,L)$
is one-dimensional and we fix an element $\det\Psi\in\Lambda^{N}H^{0}(X,L).$
We may identify $\det\Psi$ with a holomorphic section of $L^{\boxtimes N}$
over the $N-$fold product $X^{N},$ using the natural embedding \[
\Lambda^{N_{k}}H^{0}(X,L)\hookrightarrow H^{0}(X,L)^{\otimes N}\simeq H^{0}(X^{N_{k}}L^{\boxtimes N})\]
Now we may define the probability measure $\mu^{(N)}$ on $X^{N}$
by \[
\mu^{(N)}:=\frac{\left\Vert \det\Psi\right\Vert ^{2}}{\mathcal{Z}}\nu^{\otimes N}\]
where the point-wise norm is computed wrt an Hermitian metric on $L$
whose curvature form is $\omega$ and where the normalizing constant
is the $L^{2}-$norm of $\det\Psi$ induced by the weighted measure
$(\nu,\left\Vert \cdot\right\Vert ):$ \[
\mathcal{Z}:=\left\Vert \det\Psi\right\Vert _{L^{2}(X^{N},\nu^{\otimes N})}^{2}\]
By homogeneity $\mu^{(N)}$ is invariant under scaling of $\left\Vert \cdot\right\Vert $
and hence it only depends on $(\mu,\omega).$ Now the whole sequence
$\mu^{(N_{k})}$ (and the corresponding normalization constants $\mathcal{Z}_{k})$
is defined by replacing $L$ with its $k$th tensor power $kL$ and
using the induced norms. The constant $\mathcal{Z}_{k}$ depends multiplicatively
on the choice of generator $\det\Psi_{k}\in\Lambda^{N_{k}}H^{0}(X,kL)$
but, by homogeneity, the corresponding probability measure $\mu^{(N_{k})}$
does not. 

To obtain a concrete formula for the probability measure $\mu^{(N_{k})}$
we may fix a base $\Psi_{1}^{(k)},...,\Psi_{N}^{(k)}$ and note that
\begin{equation}
(\det\Psi)(x_{1},...,x_{N_{k}})=\det(\Psi_{i}^{(k)}(x_{j}))\in L_{x_{1}}\otimes\cdots\otimes L_{x_{N}}\label{eq:slater det}\end{equation}
 which may be locally represented by a local holomorphic function
$f_{k}$ on $X^{N_{k}}.$ Hence, $\mathcal{Z}_{k}$ may be written
as \[
\mathcal{Z}_{k}=\int_{X^{N_{k}}}|(f_{k}(x_{1},...,x_{k})|^{2}e^{-k(\Phi(x_{1})+\cdots+\Phi(x_{N_{k}})}(d\nu)^{\otimes N_{k}}\]
 where $\Phi$ is the local weight of the metric on $L$ whose curvature
form is $\omega,$ i.e locally $\omega=dd^{c}\Phi.$ 

We will fix an auxiliary weighted measure $(\nu_{0},\omega_{\phi_{0}})$
(or rather $(\nu_{0},\phi_{0})$) where $\nu_{0}$ has the Bernstein-Markov
property wrt $(K_{0},\omega_{\phi_{0}})$ and take the fixed base
$\Psi_{k,1},...,\Psi_{k,N}$ above to be orthonormal wrt the inner
product on $H^{0}(X,kL)$ induced by $(\nu_{0},\phi_{0})$ and write
\begin{equation}
\tilde{\mu}^{(N_{k})}:=\left\Vert \det\Psi_{k}\right\Vert ^{2}\nu^{\otimes N}\label{eq:non-normalied measure for point pr}\end{equation}
for the corresponding \emph{non-normalized} measure on $X^{N_{k}}.$
The point is that this this will make sure that the large $k$ limit
of $\frac{1}{k^{n+1}}\log\mathcal{Z}_{k}$ exists. For example, if
$(\nu_{0},\omega_{\phi_{0}})$ coincides with $(\nu,\omega)$ then
$\mathcal{Z}_{k}=N_{k}!$ (see for example \cite{b-b}).

Given a continuous function $\phi$ on $X$ we will use the notation
$\mu_{k\phi}^{(N_{k})}$ for the probability measure on $X^{N_{k}}$
obtained by replacing $\omega$ with $\omega_{\phi}.$ Equivalently,
this means that the point-wise norms $\left\Vert \cdot\right\Vert ^{2}$
are replaced by $\left\Vert \cdot\right\Vert ^{2}e^{-k\phi(\cdot)}$
and hence $\mu_{k\phi}^{(N_{k})}$ can be written as the {}``tilted''
probability measure \[
\mu_{k\phi}^{(N_{k})}=\frac{1}{\mathcal{Z}_{k\phi}}\mu^{(N_{k})}e^{-k\phi}\]
 where $\phi$ is the corresponding \emph{linear statistic} $\phi(x_{1})+....+\phi(x_{N_{k}})$
on $X^{N_{k}}$  and \[
\mathcal{Z}_{k\phi}=\int_{X^{N_{k}}}\mu^{(N_{k})}e^{-k\phi}\]
We will denote by $K$ the support of $\nu$ so that the measure $\mu^{(N_{k})}$
defines a probability measure on $K^{N_{k}}.$ In the following we
will use the same notation $\mu^{(N_{k})}$ for the density on $K^{N_{k}}$
of $\mu^{(N_{k})}$ w.r.t. the measure $\nu^{\otimes N_{k}};$ the
precise meaning will hopefully be clear from the context.

\subsection{Definition of a large deviation principle (LDP)}

Let us recall the general definition of a LDP due to Donsker and Varadhan
(see for example the book \cite{d-z}):
\begin{defn}
\label{def:large dev}Let $\mathcal{P}$ be a Polish space, i.e. a
complete separable metric space.

$(i)$ A function $I:\mathcal{\, P}\rightarrow[0,\infty]$ is a \emph{rate
function} iff it is lower semi-continuous. It is a \emph{good} \emph{rate
function} if it is also proper.

$(ii)$ A sequence $\Gamma_{k}$ of probability measures on $\mathcal{P}$
satisfies a \emph{large deviation principle} with \emph{speed} $r_{k}$
and \emph{rate function} $I$ iff

\[
\limsup_{k\rightarrow\infty}\frac{1}{r_{k}}\log\Gamma_{k}(\mathcal{F})\leq-\inf_{\mu\in\mathcal{F}}I(\mu)\]
 for any closed subset $\mathcal{F}$ of $\mathcal{P}$ and \[
\liminf_{k\rightarrow\infty}\frac{1}{r_{k}}\log\Gamma_{k}(\mathcal{G})\geq-\inf_{\mu\in G}I(\mu)\]
 for any open subset $\mathcal{G}$ of $\mathcal{P}.$ 
\end{defn}
Let now $\mathcal{P=P}(K)$ be the space of all probability measures
on $X$ which is a Polish space, where the topology corresponds to
the weak convergence of measures.

Given a set $\mathcal{F}$ in $\mathcal{P}(K)$ we will write \[
K^{N}\cap\mathcal{F}:=(\delta_{N}/N)^{-1}(\mathcal{F})\]
 where $\delta_{N}$ denotes the natural inclusion \ref{eq:intro random measure}
of $K^{N}$ into $\mathcal{P}(K)$ (i.e. the map defined by the empirical
measure)

\subsection{\label{sub:A-direct-proof}A direct proof of Theorem \ref{thm:intro large dev} }

\subsubsection{Preliminaries on asymptotics of $\mu_{k\phi}^{(N_{k})}$}

First we recall the results in the following
\begin{thm}
\label{thm:[asym-Fekete]} Let $K$ be non-pluripolar subset of $X$,
$\phi$ a continuous function on $X$ and $\nu$ a probability measure
on $X$ which has the Bernstein-Markov property w.r.t. 
\begin{itemize}
\item \cite{b-b} Given a reference weighted Bernstein-Markov measure $(\nu_{0},\phi_{0})$
the following convergence holds: \textup{\begin{equation}
k^{-(n+1)}\log\left\Vert \det\Psi_{k}\right\Vert _{L^{\infty}(k\phi,K^{N_{k}})}^{2}\rightarrow-\mathcal{E}(P_{K}\phi)+\mathcal{E}(P_{K_{0}}\phi_{0})\label{eq:conv trans diam}\end{equation}
} (where the norms are computed wrt the metric on $L$ whose curvature
current is $\omega)$ and if the measure $\nu$ has the Bernstein-Markov
property wrt $(K,\phi),$ then we also have \begin{equation}
k^{-(n+1)}\log\left\Vert \det\Psi_{k}\right\Vert _{L^{2}(k\phi,K^{N_{k}},\nu)}^{2}\rightarrow-\mathcal{E}(P_{K}\phi)+\mathcal{E}(P_{K_{0}}\phi_{0})\label{eq:conv free en not intro}\end{equation}

\item \cite{b-b-w} Let $\mathbf{(x}_{k})$ be a sequence of configurations
in $K$ (i.e. $\mathbf{x}_{k}\in K^{N_{k}})$ such that \[
\liminf_{k\rightarrow\infty}k^{-(n+1)}\log\mu_{k\phi}^{(N_{k})}(\mathbf{x}_{k})\geq0\]
 Then $\mu_{k}:=j_{N_{k}}\mathbf{(x}_{k})$ converges weakly to the
equilibrium measure $MA(P_{(K,\omega)}\phi).$
\item \cite{b-b-w} If the measure $\nu$ has the Bernstein-Markov property
wrt $(K,\phi),$ \[
(\E_{k\phi}(\frac{\delta_{N}}{N})=)\int_{K^{N-1}}\mu_{k\phi}^{(N_{k})}\rightarrow MA(P_{(K,\omega)}\phi)\]
weakly.
\end{itemize}
\end{thm}
In section \ref{sub:Proof-of-the-non-det ldp} we will repeat the
simple argument used in \cite{b-b} to deduce \ref{eq:conv free en not intro}
from \ref{eq:conv trans diam} in the previous theorem. 

We will also need a localized version of the last two points in the
previous theorem. To this end, define the following set:\begin{equation}
A_{k\phi}:=\{\mathbf{x}_{k}\in K^{N_{k}}:\,\, k^{-(n+1)}\log\mu_{k\phi}^{(N_{k})}(\mathbf{x}_{k})\geq-1/k\}\label{eq:the set a}\end{equation}

\begin{lem}
\label{lem:localization}Let $\phi$ be a continuous function on $X.$
Then \[
\liminf_{k\rightarrow\infty}k^{-(n+1)}\log(\int_{A_{k\phi}}\left\Vert \det\Psi_{k}\right\Vert ^{2}d\nu^{N_{k}})\geq\mathcal{E}(P_{K_{0}}\phi_{0})-\mathcal{E}(P_{K}\phi)+\int_{X}\phi MA(P_{K}\phi)\]
\end{lem}
\begin{proof}
First observe that, after normalization, we may assume that $\mathcal{Z}_{k}=1$
(which by $(ii)$ corresponds to having $\mathcal{E}(P_{K_{0}}\phi_{0})-\mathcal{E}(P_{K}0)=0)$
Decomposing the point-wise norm $\left\Vert \det\Psi_{k}\right\Vert ^{2}=\left\Vert \det\Psi_{k}\right\Vert _{k\phi}^{2}e^{k\phi}$
and using Jensen's inequality applied to the convex function $e^{t}$
gives \[
(\int_{A_{k}}\left\Vert \det\Psi_{k}\right\Vert _{k\phi}^{2}e^{k\phi}d\nu^{N_{k}})\geq\mathcal{Z}_{k\phi}'\textrm{exp \ensuremath{(}}\int_{A_{k\phi}}\frac{\left\Vert \det\Psi_{k}\right\Vert _{k\phi}^{2}}{\mathcal{Z}_{k\phi}}kud\nu^{N_{k}})\]
 where \[
\mathcal{Z}_{k\phi}':=(\int_{A_{k\phi}}\left\Vert \det\Psi_{k}\right\Vert _{k\phi}^{2}d\nu^{N_{k}})\]
 Hence, the sequence in the r.h.s in the statement of the lemma is
bounded from below by \[
k^{-(n+1)}\log\mathcal{Z}_{k\phi}+k^{-n}\textrm{\ensuremath{(}}\int_{A_{k\phi}}\frac{\left\Vert \det\Psi_{k}\right\Vert _{k\phi}^{2}}{\mathcal{Z}_{k\phi}}\phi d\nu^{N_{k}})(\mathcal{Z}_{k\phi}/\mathcal{Z}_{k\phi}')\]
 \[
+k^{-(n+1)}\log(\mathcal{Z}_{k\phi}/\mathcal{Z}_{k\phi}')\]
 But by the {}``exponential'' decay of the probability measure $\mu_{k\phi}^{(N_{k})}$
on the complement of $A_{k\phi}:$ \begin{equation}
\mathcal{Z}_{k\phi}/\mathcal{Z}_{k\phi}'\rightarrow1\label{eq:quotient conv to one}\end{equation}
 Indeed, \[
\mathcal{Z}_{k}'/\mathcal{Z}_{k}=\int_{A_{k}}\mu_{k\phi}^{(N_{k})}=1-\int_{K^{N_{k}}-A_{k}}\mu_{k\phi}^{(N_{k})}\]
 and on $K^{N_{k}}-A_{k}$ we have, by definition, $\mu_{k\phi}^{(N_{k})}<e^{-k^{n}}d\nu^{N_{k}}$
proving the convergence \ref{eq:quotient conv to one}. Finally, using
$(ii)$ and $(iii)$ in Theorem \ref{thm:[asym-Fekete]} combined
with the exponentially decay of $\mu_{k\phi}^{(N_{k})}(=\frac{\left\Vert \det\Psi_{k}\right\Vert _{k\phi}^{2}}{\mathcal{Z}_{k\phi}})$
on the complement of $A_{k\phi}$ finishes the proof of the lemma. 
\end{proof}

\subsection{Proofs of upper and lower bounds in Theorem \ref{thm:intro large dev}}

To simplify the notation we assume that $V=1.$ First we will prove
the upper bound of the theorem (without the normalization factor).
It does not use the Bernstein-Markov property of the measure $\nu.$
It will be convenient to first establish the LDP for the non-normalized
measures $\tilde{\mu}^{(N_{k})}$ (formula \ref{eq:non-normalied measure for point pr}). 
\begin{prop}
\label{pro:upper bound over closed st}Assume that $\nu$ is a probability
measure supported on the compact set $K$ in $X$ and let $\mathcal{F}$
be a closed set in $\mathcal{P}(K).$ Then \[
\limsup_{k\rightarrow\infty}k^{-(n+1)}\log(\left\Vert \det\Psi_{k}\right\Vert _{L^{2}(\nu,K^{N_{k}}\cap\mathcal{F})}^{2}\leq-\inf_{\mu\in\mathcal{F}}E_{\omega}(\mu)+\mathcal{E}_{\omega}(P_{(K_{0},\omega)}\phi_{0})\]
\end{prop}
\begin{proof}
We may assume that $K$ is not pluri-polar; otherwise the right hand
side is infinite \cite{bbgz}. Since $\nu$ is a probability measure
supported on $K$ we have \begin{equation}
(kN_{k})^{-1}\log(\left\Vert \det\Psi_{k}\right\Vert _{L^{2}(\nu,K^{N_{k}}\cap\mathcal{F})}^{2}\leq(kN_{k})^{-1}\log(\left\Vert \det\Psi_{k}\right\Vert _{L^{\infty}(K^{N_{k}}\cap\mathcal{F})}^{2}.\label{eq:prop upper bound closed: triv}\end{equation}
 Let $(x_{1},...,x_{N_{k}})\in K^{N_{k}}\cap\mathcal{F}$ be a configuration
realizing the sup in the r.h.s. above and fix a continuous $\omega-$psh
function $\phi$ on $X.$ Then the r.h.s above may be written as $(kN_{k})^{-1}\log(\left\Vert \Psi_{k}\right\Vert ^{2}(x_{1},...,x_{N_{k}}))=$\begin{equation}
=(kN_{k})^{-1}\log(\left\Vert \det\Psi_{k}\right\Vert _{k\phi}^{2}(x_{1},...,x_{N_{k}}))+\frac{1}{N_{k}}\sum_{i=1}^{N_{k}}\delta_{x_{i}}\phi\leq\label{eq:prop upper bound over closed s}\end{equation}
 \[
\leq(kN_{k})^{-1}\log(\left\Vert \det\Psi_{k}\right\Vert _{L^{\infty}(k\phi,K^{N_{k}})}^{2}+\frac{1}{N_{k}}\sum_{i=1}^{N_{k}}\delta_{x_{i}}\phi\]
 After passing to a subsequence we may assume, by weak compactness,
that \[
\frac{1}{N_{k}}\sum_{i=1}^{N_{k}}\delta_{x_{i}}\rightarrow\mu\in\mathcal{F}\]
 weakly, since $\mathcal{F}$ is closed. In particular, since $\phi$
is a continuous function on $X$ it follows that\begin{equation}
(kN_{k})^{-1}\log(\left\Vert \det\Psi_{k}\right\Vert _{L^{\infty}(k\phi,K^{N_{k}})}^{2}+\frac{1}{N_{k}}\sum_{i=1}^{N_{k}}\delta_{x_{i}}\phi\rightarrow\mathcal{E}_{\omega}(P_{K_{0}}\phi_{0})-\mathcal{E}_{\omega}(P_{K}\phi)+\int_{X}\phi\mu,\label{eq:prop upper bound over closed: fek}\end{equation}
 using \ref{eq:conv trans diam}. Since this holds for any such $\phi$
combining \ref{eq:prop upper bound over closed s} and \ref{eq:prop upper bound over closed: fek}
gives, also using \ref{eq:s in terms of cont} in Proposition \ref{pro:prop of s},
\begin{equation}
\limsup_{k\rightarrow\infty}(kN_{k})^{-1}\log(\left\Vert \det\Psi_{k}\right\Vert ^{2}(x_{1},...,x_{N_{k}}))\leq-E_{\omega}(\mu)+\mathcal{E}_{\omega}(P_{K_{0}}\phi_{0}))\label{eq:proof prop upper bd a}\end{equation}
 for the chosen subsequence of configurations. Hence, by \ref{eq:prop upper bound closed: triv}
\[
\limsup_{k\rightarrow\infty}(kN_{k})^{-1}\log(\left\Vert \det\Psi_{k}\right\Vert _{L^{2}(\nu,K^{N_{k}}\cap\mathcal{F})}^{2}\leq\sup_{\mu\in\mathcal{F}}(-E_{\omega}(\mu))+\mathcal{E}_{\omega}(P_{K_{0}}\phi_{0}))\]
 which finishes the proof of the proposition. 
\end{proof}
Finally, we will prove the following lower bound: 
\begin{prop}
Suppose that the measure $\nu$ has the Bernstein-Markov property
wrt the set $K$ in $X.$ Then for any open set $\mathcal{G}$ in
$\mathcal{P}(K)$\[
\liminf_{k\rightarrow\infty}k^{-(n+1)}\log\left\Vert \det\Psi_{k}\right\Vert _{L^{2}(\nu,K^{N_{k}}\cap\mathcal{G})}^{2}\geq-\inf_{\mu\in\mathcal{G}}E_{\omega}(\mu)+\mathcal{E}_{\omega}(P_{(K_{0},\omega)}\phi_{0})\]
\end{prop}
\begin{proof}
For a given $\mu\in\mathcal{G}$ we have to prove \begin{equation}
\liminf_{k\rightarrow\infty}k^{-(n+1)}\log\left\Vert \det\Psi_{k}\right\Vert _{L^{2}(\nu,K^{N_{k}}\cap\mathcal{G})}^{2}\geq-E_{\omega}(\mu)+\mathcal{E}_{\omega}(P_{(K_{0},\omega)}\phi_{0})\label{eq:pf lower bound: ineq to prove}\end{equation}
 We may assume that $E(\mu)<\infty$ (otherwise the statement is trivially
true). But then Theorem \ref{thm:var sol of ma} gives that there
exists $\phi_{\mu}\in\mathcal{E}^{1}(X,\omega)$ such that $MA(\phi_{\mu})=\mu.$
To fix ideas we first assume that $\phi_{\mu}$ is continuous. By
\ref{eq:pf lower bd: a set}\[
\liminf_{k\rightarrow\infty}k^{-(n+1)}\log\left\Vert \det\Psi_{k}\right\Vert _{L^{2}(\nu,K^{N_{k}}\cap\mathcal{G})}^{2}\geq\liminf_{k\rightarrow\infty}k^{-(n+1)}\log\left\Vert \det\Psi_{k}\right\Vert _{L^{2}(\nu,K^{N_{k}}\cap A_{k\phi_{\mu}})}^{2}\geq\]
 \[
\geq\left(-\mathcal{E}(P_{K}\phi_{\mu})+\int\phi_{\mu}MA(P_{K}\phi_{\mu})\right)+\mathcal{E}_{\omega}(P_{K_{0}}\phi_{0})\]
 using lemma \ref{lem:localization} in the last step. Since, $\phi_{\mu}$
is assumed continuous $P_{K}\phi_{\mu}=\phi_{\mu}$ almost everywhere
wrt $MA(P_{K}\phi_{\mu})$ (by the orthogonality relation \ref{eq:og relatgion})
and hence the first terms above equals $-E(\mu),$ proving the desired
bound. Finally, in the general case when $\phi_{\mu}$ is a general
potential of finite energy we take a sequence $\phi_{j}$ of continuous
$\omega-$psh functions decreasing to $\phi_{\mu}.$ By Lemma \ref{lem:appr}
\[
\mu_{j}:=MA(P_{K}\phi_{j})\rightarrow MA(\phi_{\mu})=\mu\]
 in $\mathcal{P}(K).$  In particular, since $\mathcal{G}$ is assumed
open in $\mathcal{P}(K),$\begin{equation}
\mu_{j}\subset\mathcal{G}\label{eq:pf lower bd}\end{equation}
 for $j>>1.$ Next, fix a large index $j$ and consider the set $A_{k\phi_{j}},$
defined as in formula \ref{eq:the set a}. Then for $k>>1$\begin{equation}
A_{k\phi_{j}}\subset(\frac{\delta_{N_{k}}}{N_{k}})^{-1}\mathcal{G}\label{eq:pf lower bd: a set}\end{equation}
 Indeed, assume for a contradiction that the previous statement is
false. Then there is a sequence $\mathbf{(x}_{k_{i}})$ of configurations
$\mathbf{x}_{k_{i}}\in K^{N_{k}}-(\frac{\delta_{N_{k}}}{N_{k}})^{-1}\mathcal{G}$
such that\[
\liminf_{k_{i}}k_{i}^{-(n+1)}\log(\gamma_{k_{i}\phi_{j}}(\mathbf{x}_{k_{i}})\geq0\]
 But then Theorem \ref{thm:[asym-Fekete]} gives that \[
\mu_{\mathbf{x}_{k_{i}}}\rightarrow\mu_{j}\]
 in $\mathcal{P}(K),$ forcing $\mu_{j}\in\mathcal{P}(K)-\mathcal{G},$
which contradicts \ref{eq:pf lower bd}. We may now repeat the previous
argument with $\phi_{\mu}$ replaced by $\phi_{j}$ for $j>>1$ and
instead get the lower bound \[
\left(-\mathcal{E}(P_{K}\phi_{j})+\int\phi_{j}MA(P_{K}\phi_{j})\right)+\mathcal{E}_{\omega}(P_{K_{0}}\phi_{0})=-E(\mu_{j})+\mathcal{E}_{\omega}(P_{K_{0}}\phi_{0})\]
Finally, letting $j$ tend to infinity and using the convergence in
Lemma \ref{lem:appr} concludes the proof in the general case. 
\end{proof}
Finally, note that to obtain the rate functional $H_{(K,\omega)}$
for the LDP wrt the normalized measures $\mu^{(N_{k})}$ we just have
to normalize by dividing by $\mathcal{Z}_{k}$ which, by the first
point in Thm \ref{thm:[asym-Fekete]} gives the rate functional \[
H_{(K,\omega)}(\mu):=(E_{\omega}(\mu))+\mathcal{E}_{\omega}(P_{K_{0}}\phi_{0})-\mathcal{E}_{\omega}(P_{K}0)-\mathcal{E}_{\omega}(P_{K_{0}}\phi_{0})=(E_{\omega}(\mu))-\mathcal{E}_{\omega}(P_{K}0)\]
 which coincides with the definition in formula \ref{eq:rate functional as energies}
of $H_{(K,\omega)}.$ This completes the proof of the theorem.
\begin{rem}
Applying the large deviation principle established above (for the
sequence of probability measures) to $\mathcal{F}=\mathcal{G}=\mathcal{P}(K)$
gives \[
\log(1)=0=\inf_{\mu\in\mathcal{P}(K)}(E_{\omega}(\mu))+\mathcal{E}_{\omega}(P_{K}0)-\mathcal{E}_{\omega}(P_{K_{0}}\phi_{0}),\]
 which proves the formula \ref{eq:capacity as inf}.
\end{rem}

\subsection{\label{sub:using g-e}Proof of Theorem \ref{thm:intro large dev}
using the Gärtner-Ellis theorem}

For the proof and references for the following abstract version of
the Gärtner-Ellis theorem see \cite{d-z} (Cor 4.6.14, p. 148)
\begin{thm}
\label{thm:(Abstract-G=0000E4rtner-Ellis-Theorem}(Abstract Gärtner-Ellis
Theorem ). Let $\mathcal{M}$ be a locally convex Hausdorff topological
vector space and $\Gamma_{k}$ a sequence of Borel measures on $\mathcal{M}$
which is exponentially tight. Assume that there is a sequence of positive
numbers $r_{k}$ such that the Laplace transforms $\widehat{\Gamma}_{k},$
seen as functionals on the dual $\mathcal{M}^{*},$ satisfy \[
\frac{1}{r_{k}}\widehat{\Gamma}_{k}[r_{k}u]\rightarrow\Lambda[u]\]
 for any $u$ in $\mathcal{M}^{*}$ where the functional $\Lambda$
is Gateau differentiable on $\mathcal{M}^{*}.$ Then $\Gamma_{k}$
satisfies a LDP with speed $r_{k}$ and with a rate functional $H:=\Lambda^{*}$
on $\mathcal{M},$ i.e. $H$ is the Legendre-Fenchel transform of
$\Lambda.$
\end{thm}
To apply this theorem to the present setting we let $\mathcal{M}(=\mathcal{M}(K))$
be the space of all signed finite Borel measures on $K$ with its
the usual weak topology, i.e. $\mu_{j}\rightarrow\mu$ iff \[
\left\langle u,\mu_{j}\right\rangle :=\int_{X}u\mu_{j}\rightarrow\int_{X}u\mu\]
 for any continuous function $u,$ i.e. for all $u\in C^{0}(K).$
As is well-known $\mathcal{M}$ is a locally convex Hausdorff topological
vector space and it is the topological dual of the vector space $C^{0}(K).$
Moreover, since $K$ is compact so is $\mathcal{M}$ and hence the
tightness condition in the theorem is automatic (see the definition
in \cite{d-z}). We let $\Gamma_{k}$ be the laws of the empirical
measure of the determinant process defined above: \[
\Gamma_{k}:=\frac{1}{N_{k}}\delta_{N_{k}}(\mu^{(N_{k})})\]
 The Laplace transform of $\Gamma_{k}$ is defined as the following
functional $\widehat{\Gamma}_{k}$ on the dual $\mathcal{M}^{*}=C^{0}(K):$\[
\widehat{\Gamma}_{k}[u]:=\int_{\mathcal{M}}d\Gamma_{k}(\mu)e^{\left\langle u,\mu\right\rangle }\]
which in the present setting can be written as \[
\widehat{\Gamma}_{k}[u]=\log\E(e^{\frac{1}{N_{k}}u(x_{1})+...})\]
by pulling back the integral to above to $K^{N}.$ By Theorem A and
Lemma 4.3 in \cite{b-b} we get (just as in the previous section)
\[
\frac{1}{kN_{k}}\widehat{\Gamma}_{k}[kN_{k}u]\rightarrow\Lambda[u]:=-(\mathcal{E}\circ P_{(K,\omega})(-u)+\mathcal{E}_{\omega}(P_{(K,\omega)}0)\]
 if $\nu$ has the Bernstein-Markov property wrt $(K,\omega_{u}).$
In other words, the convergence holds for all $u$ if $\nu$ has the
Bernstein-Markov property wrt $K$ (see section \ref{sub:Bernstein-Markov-measures}).
Moreover, by Theorem B in \cite{b-b} (stated as Theorem \ref{thm:diff thm}
in the present paper) $\Lambda$ is Gateaux differentiable. All in
all this means that we can apply the Gärtner-Ellis theorem and deduce
that an LDP holds with rate functional \[
H(\mu)=\Lambda^{*}(\mu):=\sup_{u\in C^{0}(X)}(\Lambda(u)-\left\langle u,\mu\right\rangle )\]
To conclude the proof of Theorem \ref{thm:intro large dev} we set
$\phi=-u$ above and then we just have to verify that $H(\mu)=\infty$
if $\mu$ is not a probability measure of finite energy. But since
the differential of $\Lambda$ is always a probability measure (by
Theorem \ref{thm:diff thm}), i.e. the image of the {}``gradient
map'' defined by $\Lambda$ is contained in $\mathcal{P}(K),$ this
follows by general convexity theory.

\subsection{\label{sub:Remarks-on-normalizations}Remarks on normalizations of
rate functionals, energy and Vandermonde determinants}

As shown above the LDP wrt the non-normalized measures $\left\Vert \det\Psi_{k}\right\Vert ^{2}\nu^{\otimes N}$
has a rate functional \begin{equation}
\tilde{E}_{\omega}(\mu):=E_{\omega}(\mu)-\mathcal{E}_{\omega}(P_{(K_{0},\omega)}\phi_{0})\label{eq:non-norm rate f}\end{equation}
where the constant $\mathcal{E}_{\omega}(P_{(K_{0},\omega)}\phi_{0})$
is independent (as it must) of the support $K$ of $\nu.$ It follows
immediately from the LDP that \begin{equation}
\tilde{E}_{\omega+dd^{c}\phi}(\mu):=\tilde{E}_{\omega}(\mu)+\int_{X}\phi\mu,\label{eq:transf of en with tild}\end{equation}
 which of course could also be proved directly using the explicit
expression \ref{eq:non-norm rate f}). Let us illustrate this in the
case of multivariate polynomials ensembles (section \ref{sec:Examples}).
We fix $\omega_{0}$ as in the beginning of section \ref{sec:Examples}
and choose the reference measure to be$(\nu_{0},0)$ where $\nu_{0}$
is the invariant probability measure on the unit-torus in $\C^{n}.$
Then the base $(\Psi_{k,i})$ can be taken as multinomials and $\det\Psi_{k}=\Delta^{(N_{k})}(z_{1},...,z_{N_{k}})$
is then the multivariate \emph{Vandermonde determinant} (as in section
\ref{sec:Examples}). The LDP in the corresponding \emph{weighted}
setting can now be symbolically written as \begin{equation}
|\Delta^{(N_{k})}(z_{1},...,z_{N_{k}})|^{2}e^{-k(\Phi(z_{1})+\cdots)}\sim e^{-\frac{1}{n!}k^{n+1}(\tilde{E}_{0}(\mu)+\int\Phi\mu)}\label{eq:rate f for vandermonde in ct case}\end{equation}
where the {}``normalized energy'' $\tilde{E}_{0}(\mu)$ is independent
of $\Phi$ and $K$ (by the transformation property \ref{eq:transf of en with tild}).
In the classical case when $n=1$ it is not hard to check that $\tilde{E}_{0}(\mu)$
is the classical logarithmic energy of a measure $\mu:$ \begin{equation}
\tilde{E}_{0}(\mu)=-\int_{\C}\log|z-w|\mu(z)\otimes\mu(w)\label{eq:unweighted energy as log energy n is one}\end{equation}
Indeed, taking $\omega$ to vanish on a neighborhood of the support
of $\mu$ and using the Green function expression \ref{eq:energy as green}
shows that $\tilde{E}_{0}(\mu)=-\int_{\C}\log|z-w|\mu(z)\otimes\mu(w)+C.$
To see that $C=0$ we take $\mu=\omega_{0}$ to be the invariant measure
on $S^{1},$ i.e. $\mu=dd^{c}\Phi_{0}$ for $\Phi_{0}=\log^{+}|z|^{2}$
so that $\tilde{E}_{0}(\mu)=0+C.$ Since $\Phi_{0}=0$ on the support
of $\mu=\omega_{0}$ we can then use the formula \ref{eq:non-norm rate f}
with $\omega=dd^{c}\Phi_{0}$ which gives $\tilde{E}_{0}(\mu)=\tilde{E}_{\omega_{0}}(\mu)=E_{\omega_{0}}(\omega_{0})-\mathcal{E}_{\omega_{0}}(P_{(S^{1},\omega_{0})}0)=0-0$
using that $P_{(S^{1},\omega_{0})}0=0$ (by the maximum principle).
All in all this forces $C=0$ showing that \ref{eq:unweighted energy as log energy n is one}
holds. Hence the rate functional in \ref{eq:rate functional as energies}
is, when $n=1,$ precisely the \emph{weighted logarithmic energy}
of $\mu$ (which is the subject of the book \cite{s-t}). In physical
terms $\Phi$ hence acts as an exterior potential. In fact, the weighted
energy appearing in the rate functional in \ref{eq:rate f for vandermonde in ct case}
can be extended to the setting when $K$ is \emph{non-compact }as
explained in the following section.

\subsection{\label{sub:Proof-ldp-non-cpt}Proof of the LDP for non-compact $K$ }

In this section we will obtain a variant of the LDP which applies
to non-compact sets $K$ and in particular to $K=\R^{n}$ or $K=\C^{n}.$
We will consider the following general setting. Starting with an open
set $U$ of $X$ such that $X-U$ is locally pluripolar (in the applications
that we have in mind $X-U$ will even be an analytic subvariety).
We will say that a pair $(\nu,\phi)$ of a measure $\nu$ on $U$
and a continuous function $\phi$ is \emph{admissible} if 
\begin{itemize}
\item $\phi\rightarrow\infty$ at infinity in $U$ (i.e. $\phi$ is proper
on $U)$
\item $\int e^{-k\phi}\nu<\infty$ for $k\geq k_{0}.$
\end{itemize}
We will denote the support of $\nu$ in $U$ by $K,$ which is a closed
set in $U$ (but possibly non-compact in $X$!). In the following
we fix a continuous metric $\left\Vert \cdot\right\Vert $ on $L\rightarrow X$
with normalized curvature form $\omega_{0}.$ Then \[
\mu_{k\phi}^{(N_{k})}:=\left\Vert \det\Psi_{k}\right\Vert ^{2}e^{-k\phi}\nu^{\otimes N}/\mathcal{Z}_{k}\]
 is a well-defined probability measure on $K^{N_{k}}$ for $k\geq k_{0}.$ 

The next theorem gives a LDP which is a variant of Theorem \ref{thm:intro large dev}.
It is formulated in terms of the standard weak topology on the space
$\mathcal{M}(K)$ of all signed measure on $K,$ induced by the dual
$C_{b}(K$) consisting of all functions $u$ on $K$ which are continuous
and bounded (recall that $K$ is not assumed to be compact!). 
\begin{thm}
Let $(\nu,\phi)$ be an admissible pair such that $(\nu,\phi+u)$
satisfies the BM-property \ref{eq:bm with phi} for any $u\in C_{b}(K)$
and such that that the support $K$ of $\nu$ is non-pluripolar. Then
the laws on $\mathcal{P}(K)$ of the probability measure $\mu_{k\phi}^{(N_{k})}$
on $K^{N_{k}}$ satisfy a \emph{large deviation principle} (LDP) with
a \emph{good rate functional $H(=H_{(K,\omega_{0},\phi)})$} and speed
$Vk^{n+1}.$ On the space $\mathcal{P}(K)$ the rate functional \emph{$H(\mu)$}
is minimized (and vanishes) precisely on the pluripotential equilibrium
measure $\mu_{eq}(:=MA(P_{(K,\omega_{0})}\phi).$ Moreover, the rate
functional may be decomposed as \begin{equation}
H(\mu)=E_{\omega_{0}}(\mu)+\int\phi\mu-C\label{eq:decomp of weighted energ-1}\end{equation}
where $C$ is a constant (depending on $(K,\omega_{0}+dd^{c}\phi)).$
\end{thm}
Before starting the proof it will be convenient to make the additional
(very weak) assumption that $K$ be regular in the sense that $P_{K}(\phi+u)\in C^{0}(X)$
if $u\in C_{b}^{0}(K).$ This assumption may be removed by approximation
just as in the proof of Theorem A in \cite{b-b} (and anyway it is
automatically satisfied in the main cases $K=\R^{n}$ or $K=\C^{n}$
considered below). To simplify the notation we will often omit the
subscript $\omega$ in $P_{(K,\omega)}$ and $\mathcal{E}_{\omega}$
and simply write $P_{K}$ and $\mathcal{E},$ respectively.

\subsubsection{Localization to a {}``ball'' $B_{R}$}

Let $\rho$ be a given proper function on $U$ and write $B_{R}:=\{\rho\leq-R\}$
so that $B_{R}$ is sequence of increasing compact sets covering $U.$
We first note that the support of the equilibrium measure $\mu_{eq}:=MA(P_{K}\phi)$
is compact in $U.$ Indeed, in general it is contained in the closed
set $D:=\{P_{K}\phi\geq\phi\}$ (this is a well-known properties of
{}``free envelopes'' and follows for example from Prop 1.10 in \cite{b-b}
or rather its proof) and since $\phi\rightarrow\infty$ at infinity
in $U$ and $P_{K}\phi\leq C$ (using $P_{K}\phi\leq P_{K\cap B_{R}}\phi)$
it follows that $D$ is compact in $U.$ Let us fix a {}``ball''
$B_{R}$ in $U$ containing the support of $\mu_{eq}.$ 
\begin{lem}
\label{lem:pf of non-cpt ldp}We have that $P_{K\cap B_{R}}\phi=P_{K}\phi.$
Moreover, for any $\Psi_{k}\in H^{0}(X,kL)$ \begin{equation}
\sup_{K}\left\Vert \Psi_{k}\right\Vert ^{2}e^{-k\mbox{\ensuremath{\phi}}}=\sup_{K}\left\Vert \Psi_{k}\right\Vert ^{2}e^{-kP_{K}\mbox{\ensuremath{\phi}}}\label{eq:sup same as projec}\end{equation}
 and for any $\epsilon>0,$ there is $C_{\epsilon}>0$ such that \begin{equation}
(\left\Vert \Psi_{k}\right\Vert ^{2}e^{-k\mbox{\ensuremath{\phi}}})(x)\leq C_{\epsilon}e^{k(P_{K}\phi-\phi)}e^{\epsilon k}\int\left\Vert \Psi_{k}\right\Vert ^{2}e^{-k\mbox{\ensuremath{\phi}}}d\nu\label{eq:exp decay in non-cpt}\end{equation}
\end{lem}
\begin{proof}
By definition $P_{K\cap B_{R}}\phi\geq P_{K}\phi$ and $P_{K\cap B_{R}}\phi\leq\phi$
on the support of $MA(P_{K}\phi).$ Since this latter set is contained
in $D$ (see above) this means that $P_{K\cap B_{R}}\phi\leq P_{K}\phi$
a.e. wrt $MA(P_{K}\phi)$ and hence the inequality holds everywhere
accord to the domination principle (see \cite{begz} for a very general
version of this principle). This shows that $P_{K\cap B_{R}}\phi=P_{K}\phi.$
Next, note that \ref{eq:sup same as projec} follows directly from
the definition of $P_{K}$ (just as in \cite{b-b}). To prove \ref{eq:exp decay in non-cpt}
we set $\psi:=\frac{1}{k}(\log(\left\Vert \Psi_{k}\right\Vert ^{2})(x)/C_{\epsilon}e^{\epsilon k}\int\left\Vert \Psi_{k}\right\Vert ^{2}e^{-k\mbox{\ensuremath{\phi}}}d\nu)$
where $C_{\epsilon}$ is chosen so that the BM-inequality holds wrt
$K,$ i.e. so that $\psi\leq\phi$ on $K.$ Since $\psi$ is a candidate
for the sup defining $P_{K}\phi$ it follows that $\psi\leq P_{K}\phi$
on $X$ which finishes the proof. 
\end{proof}
Let us first prove that the analogue of the first point in Theorem
\ref{thm:[asym-Fekete]} holds and in particular:\begin{equation}
k^{-(n+1)}\log\left\Vert \det\Psi_{k}\right\Vert _{L^{2}(k\phi,\nu^{\otimes N_{k}})}^{2}\rightarrow-\mathcal{E}_{\omega_{0}}(P_{(K,\omega_{0})}(\phi))+\mathcal{E}_{\omega_{0}}(P_{(K_{0},\omega_{0})}\phi_{0}),\label{eq:conv of free energy in non-cpt}\end{equation}
To this end we first apply the previous lemma $N_{k}$ times to $\Psi_{k}^{(j)}(z):=(\det\Psi)(x_{1},x_{2},x_{j-1}x,x_{j},...,x_{N_{k}}))$
for $j=1,..N_{k}$ giving \begin{equation}
(\left\Vert \det\Psi_{k}\right\Vert ^{2}e^{-k\mbox{\ensuremath{\phi}}})(x_{1},...,x_{N})\leq C_{\epsilon}^{N_{k}}e^{\epsilon N_{k}k}e^{-k((\phi-P_{K}\phi)(x_{1},..x_{N})}\int_{X^{N_{k}}}\left\Vert \det\Psi_{k}\right\Vert ^{2}e^{-k\mbox{\ensuremath{\phi}}}\nu^{\otimes N_{k}}\label{eq:upper bound in non-cpt}\end{equation}
Moreover, decomposing $k\phi=(k-k_{0})\phi+k_{0}\phi$ and using $P_{K}v\leq v$
on the support $K$ of $\nu$ gives\[
\int_{X^{N_{k}}}\left\Vert \det\Psi_{k}\right\Vert ^{2}e^{-k\mbox{\ensuremath{\phi}}}\nu^{\otimes N_{k}}\leq C_{\epsilon}^{N_{k}}e^{\epsilon N_{k}k}\sup_{K}(\left\Vert \det\Psi_{k}\right\Vert ^{2}e^{-kP_{K}((1-\frac{k:_{0}}{k})\phi)})(\int_{X}e^{-k_{0}\phi}\nu)^{N_{k}}\]
Using that $P_{K}$ is concave we get $P_{K}((1-\frac{k:_{0}}{k})\phi)\geq((1-\frac{k:_{0}}{k})P_{K}\phi+(\frac{k:_{0}}{k})P_{K}0$
and since $P_{K}\phi$ and $P_{K}0$ are both bounded on $X$ it follows
that \begin{equation}
\int_{X^{N_{k}}}\left\Vert \det\Psi_{k}\right\Vert ^{2}e^{-k\mbox{\ensuremath{\phi}}}\nu^{\otimes N_{k}}\leq C\sup_{K}(\left\Vert \det\Psi_{k}\right\Vert ^{2}e^{-kP_{K}\phi})C'^{N_{k}}e^{\epsilon N_{k}k}\label{eq:lower bound in non-cpt}\end{equation}
Since, by Lemma \ref{lem:pf of non-cpt ldp} $\sup_{K}\left\Vert \det\Psi_{k}\right\Vert ^{2}e^{-k\mbox{\ensuremath{\phi}}}=\sup_{K}(\left\Vert \det\Psi_{k}\right\Vert ^{2}e^{-kP_{K}\mbox{\ensuremath{\phi}}})$
combining \ref{eq:upper bound in non-cpt} and \ref{eq:lower bound in non-cpt}
(and using that $\phi\geq P_{K}\phi)$ gives \begin{equation}
k^{-(n+1)}\log\sup_{K}(\left\Vert \det\Psi_{k}\right\Vert ^{2}e^{-kP_{K}\mbox{\ensuremath{\phi}}})=k^{-(n+1)}\log\left\Vert \det\Psi_{k}\right\Vert _{L^{2}(k\phi,\nu)}^{2}+o(1)\label{eq:local of norm to ball}\end{equation}
Next, by Lemma \ref{lem:pf of non-cpt ldp} we have $P_{K}\phi=P_{K\cap B_{R}}\phi$
and hence we can apply the second point in Theorem \ref{thm:[asym-Fekete]}
to the function $P_{K\cap B_{R}}\phi$ on $X$ and deduce that \ref{eq:conv of free energy in non-cpt}
indeed holds (where we again used that $P_{K}\phi=P_{K\cap B_{R}}\phi,$
but now in the rhs in \ref{eq:conv of free energy in non-cpt}). 

Next, we note that the analogue of Theorem \ref{thm:diff thm} holds:
the functional $u\mapsto\mathcal{E}_{\omega_{0}}(P_{(K,\omega_{0})}(\phi+u)$
is Gateaux differentiable on $C_{b}(K)$ with differential $MA(P_{(K,\omega_{0})}(\phi+u)).$
In other words,

\[
\frac{\mathcal{E}_{\omega_{0}}(P_{(K,\omega_{0})}(\phi+tu)}{dt}_{t=0}=\int_{X}MA(P_{(K,\omega_{0})}\phi)u\]
Indeed, since $t$ stays in a bounded set and $u$ is bounded we may,
as explained above, assume the support of $MA(P_{(K,\omega_{0})}(\phi+tu))$
is contained in $B_{R}$ giving, just as before, that $P_{(K,\omega_{0})}(\phi+tu)=P_{(K\cap B_{R},\omega_{0})}(\phi+tu)$
and hence the differentiability follows from Theorem \ref{thm:diff thm}.

\subsubsection{\label{sub:Application-of-the-ell-g}Exponential tightness and application
of the Ellis-Gärtner theorem}

Given the previous estimates the proof could be obtained by repeating
the arguments in section \ref{sub:A-direct-proof}. But for simplicity
we will instead apply the Ellis-Gärtner theorem to this non-compact
setting. To this end we also need to verify that the corresponding
sequence $\Gamma_{k}$ is \emph{exponentially tight }(wrt the speed
of the expected LDP),\emph{ }i.e. the space $\mathcal{P}(K)$ may
be exhausted by compact subsets $\mathcal{K}_{\alpha}$ for $\alpha>0$
such that $\limsup_{k\rightarrow\infty}\log(\Gamma_{k}(\mathcal{P}(K)-\mathcal{K}_{\alpha})/kN_{k}<\alpha.$
To prove this we let $\mathcal{K}_{\alpha}$ to be the set of all
measures on $\mathcal{P}(K)$ such that $\int(\phi-P_{K}\phi)\mu\leq3\alpha.$
Since $\phi-P_{K}\phi\rightarrow\infty$ at infinity in $U,$ the
set $\mathcal{K}_{\alpha}$ is indeed compact. By definition\[
\Gamma_{k}(\mathcal{P}(K)-\mathcal{K}_{\alpha})=\int_{\{\phi-P_{K}\phi>N_{K}3\alpha\}}\frac{\left\Vert \det\Psi_{k}\right\Vert ^{2}e^{-k\mbox{\ensuremath{\phi}}})(x_{1},...,x_{N})}{\left\Vert \det\Psi_{k}\right\Vert _{L^{2}(k\phi,\nu^{\otimes N_{k}})}^{2}}\nu^{\otimes N_{k}}\]
Now, by \ref{eq:upper bound in non-cpt} the density in the previous
integral may be estimated from above by $C_{\epsilon}^{N_{k}}e^{\epsilon N_{k}k}e^{-k(\phi-P_{K}\phi)}$
for some fixed small $\epsilon>0$ (taken so that $\epsilon<\alpha/2)$
Hence, decomposing \[
e^{-k(\phi-P_{K}\phi)}=e^{-\frac{1}{2}k(\phi-P_{K}\phi)}e^{-\frac{1}{2}k(\phi-P_{K}\phi)}\leq e^{-\frac{1}{2}kN_{K}3\alpha}e^{-\frac{1}{2}k\phi}C^{k}\]
 and integrating wrt $\nu^{\otimes N_{k}}$ (and using that $\phi$
is admissible) finishes the proof of the exponential tightness.

All in all this means that we may apply the abstract Ellis-Gärtner
theorem \ref{thm:(Abstract-G=0000E4rtner-Ellis-Theorem} as before
and obtain an LDP with a rate functional expressed as a Legendre transform
\begin{equation}
H(\mu):=\left(\sup_{C_{b}(K)}\mathcal{E}_{\omega_{0}}(P_{(K,\omega_{0})}(\phi+u)-\int u\mu\right)-C\label{eq:rate funt for ldp in non-cpt as legendre}\end{equation}

\subsubsection*{Properties of the rate functional}

The existence, uniqueness and form of the minimizer of the rate functional
$H$ follows exactly as in the proof of Prop \ref{pro:rate}. The
fact that $H$ is good is a well-known consequence of the LDP and
the exponential tightness obtained above (see Lemma 1.2.18 in \cite{d-z}).
But it could also be proved directly from the the decomposition \ref{eq:decomp of weighted energ},
to whose proof we now turn. We first rewrite the first bracket in
\ref{eq:rate funt for ldp in non-cpt as legendre} as \[
\left(\sup_{\phi'\in\{\phi\}+C_{b}(K)}(\mathcal{E}_{\omega_{0}}(P_{(K,\omega_{0})}(\phi')-\int\phi'\mu)\right)+\int\phi\mu\]
Next we have, just as in the proof of Prop \ref{pro:of energy-1}
that \[
\mathcal{E}_{\omega_{0}}(P_{(K,\omega_{0})}(\phi')-\int\phi'\mu\leq\mathcal{E}_{\omega_{0}}(P_{(K,\omega_{0})}(\phi')-\int P_{(K,\omega_{0})}(\phi')\mu\leq E_{\omega_{0}}(\mu)\]
as $P_{(K,\omega_{0})}\phi'$ is a candidate for the sup defining
$E_{\omega_{0}}(\mu).$ As for the lower bound we take $\phi'_{j}$
smooth functions on $X$ increasing to the unbounded function $\phi'.$
Then $P_{(K,\omega_{0})}(\phi'_{j})\leq P_{(K,\omega_{0})}(\phi')$
and hence \[
\mathcal{E}_{\omega_{0}}(P_{(K,\omega_{0})}(\phi_{j}')-\int\phi_{j}'\mu\leq\mathcal{E}_{\omega_{0}}(P_{(K,\omega_{0})}(\phi')-\int P_{(K,\omega_{0})}(\phi')\mu+\epsilon_{j}\]
 where $\epsilon_{j}\rightarrow0$ by the monotone convergence theorem
of integration theory. Noting that $P_{(K,\omega_{0})}(\phi_{j}')=P_{(\bar{K},\omega_{0})}(\phi_{j}')$
(since by assumption $\bar{K}-K$ is locally pluripolar in $X;$ see
\cite{g-z}) we can apply Prop \ref{pro:of energy-1} to deduce that 

\[
E_{\omega_{0}}(\mu)=\sup_{\phi'\in\{\phi\}+C_{b}(K)}(\mathcal{E}_{\omega_{0}}(P_{(K,\omega_{0})}(\phi')-\int\phi'\mu)\]
finishing the proof of the decomposition formula \ref{eq:decomp of weighted energ-1}.
Note that since $\phi$ is bounded from below it then follows that
$H(\mu)<\infty$ iff $E_{\omega_{0}}(\mu)<\infty$ and $\int\phi\mu<\infty.$ 
\begin{cor}
\label{cor:ldp for vanderm with rate fu}Let $K=\R^{n}$ (or $K=\C^{n})$
and let $\nu$ be the Euclidean measure on $K.$ Assume that $\Phi$
is a function on $K$ with super logarithmic growth at infinity. Then
the push-forward $\Gamma_{k}$ of the Vandermonde measure \[
\tilde{\mu}^{(N_{k})}:=|\Delta^{(N_{k})}|^{2}e^{-k\Phi}\nu^{\otimes N_{k}}\]
 under the normalized map $\delta_{N_{k}}/N_{k}$ (formula \ref{eq:intro random measure})
satisfies a LDP with a good rate functional \begin{equation}
\tilde{E}_{\Phi}(\mu)=\tilde{E}_{0}(\mu)+\int\Phi\mu\label{eq:decomp of weighted energ}\end{equation}
with $\tilde{E}_{0}$ independent of $\Phi$ and $K.$ The rate functional
has a unique minimizer, coinciding with the pluripotential equilibrium
measure $\mu_{eq}:=MA(P_{K}\Phi).$ \end{cor}
\begin{proof}
Let $(X,L,\omega_{0}):=(\P^{n},\mathcal{O}(1),\omega_{0})$ as in
the beginning of section \ref{sec:Examples} and write $\phi=\Phi-\log^{+}|z|^{2}:=\Phi-\Phi_{0}$
on $U:=\C^{n}.$ Then we have $\left\Vert \cdot\right\Vert ^{2}e^{-k\phi}=|\cdot|^{2}e^{-k\Phi}.$
To see that the BM-property is satisfied we first note that replacing
$K$ with $K\cap B_{R}$ for an ordinary ball of radius $R$ the $(\nu,\phi)$
has the BM-property wrt $K\cap B_{R}$ for \emph{any} $\phi\in C^{0}(K)$
(and hence for $\phi+u$ as above), as is well-known \cite{bl0}.
Next we will apply the arguments in the proof of Lemma \ref{lem:pf of non-cpt ldp}:
fixing $\phi$ and defining $\psi$ as in that lemma hence gives $\psi\leq\phi$
on $B_{R}$ and in particular $\psi\leq\phi$ on the support of $MA(P_{K}\phi)$
if $R>>1.$ But then it follows from the domination principle that
$\psi\leq P_{K}\phi$ on all of $X$ and hence the inequality \ref{eq:exp decay in non-cpt}
holds (even when restricting $\nu$ to $B_{R}).$ Since, $P_{K}\phi\leq\phi$
on all of $X$ this finishes the proof of the BM-property wrt $K$.
Hence we may apply the previous theorem to obtain an LDP with a rate
functional of the form \ref{eq:decomp of weighted energ-1} for some
constant $C.$ Finally, we may simply define $\tilde{E}_{0}(\mu):=E_{\omega_{0}}(\mu)-\int\Phi_{0}\mu-C$
so that formula \ref{eq:decomp of weighted energ} holds, as desired. 
\end{proof}
When $n=1$ the functional $\tilde{E}_{0}(\mu)$ in \ref{eq:decomp of weighted energ}
can still be written in the classical form \ref{eq:unweighted energy as log energy n is one},
even if $\mu$ does not have compact support in $\C(=U)$. This follows
for example from the latter case by using the following lemma valid
in any dimension: 
\begin{lem}
Let $\mu$ be a measure such that $E(\mu)<\infty$ and write $\mu_{R}:=1_{B_{R}}$$\mu/\int_{X}\mu_{R}$,
where $B_{R}$ is a sequence of sets such that $1_{B_{R}}\mu\rightarrow\mu.$
Then $E_{\omega_{0}}(\mu_{R})\rightarrow E_{\omega_{0}}(\mu)$ as
$R\rightarrow\infty.$ \end{lem}
\begin{proof}
Let $\phi_{\mu_{R}}$ be the potential of the probability measure
$\mu_{R}$ normalized so that $\sup_{X}\phi_{\mu_{R}}=0.$ Now by
the variational property of $\phi_{\mu_{R}}$ we have \[
\mathcal{E}(\phi_{\mu_{R}})-\int\phi_{\mu_{R}}\mu\geq\left(\mathcal{E}(\phi_{\mu})-\int\phi_{\mu}\mu\right)+\delta_{R}\int\phi_{\mu_{R}}\mu\]
where $\delta_{R}\rightarrow0.$ Moreover, by Prop 3.4 in \cite{bbgz}
the following estimate holds (using that $\mu$ has finite energy):
$|\int\phi_{\mu_{R}}\mu|\leq C|\mathcal{E}(\phi_{\mu_{R}})|^{1/2}.$
Hence, the inequality above forces $-\mathcal{E}(\phi_{\mu_{R}})\leq C$
and as consequence $\phi_{\mu_{R}}$ is an asymptotic maximizer in
the sense of Theorem \ref{thm:var sol of ma}. Hence, the latter theorem
gives $\mathcal{E}(\phi_{\mu_{R}})\rightarrow\mathcal{E}(\phi_{\mu})$
and $\int\phi_{\mu_{R}}\mu\rightarrow\int\phi_{\mu}\mu.$ Using that
$\mu_{R}=1_{B_{R}}\mu(1+\delta_{R})\leq C\mu$ and dominated convergence
hence finally gives $E(\mu_{R})\rightarrow E(\mu).$
\end{proof}

\subsubsection{\label{sub:Applications-to-sections}Applications to sections vanishing
along a given hypersurface and Laplacian growth}

Let $Z$ by a smooth hypersurface in $X$ Let $H_{kZ}$ be the subspace
of $H^{0}(X,kL)$ consisting of all sections vanishing to order $k$
along $Z.$ Then any continuous Hermitian metric $\left\Vert \dot{}\right\Vert $
(with curvature form $\omega)$ and a volume form $\nu$ on $X$ induce
by restriction, an inner product on the subspace $H_{kZ}$ (that will
be non-degenerate under the assumptions below). Hence, we can associate
a sequence of determinantal point-processes to the corresponding sequence
of Hilbert spaces. We will assume that the line bundle $L-L_{Z}$
is ample, where $(L_{Z},s_{Z})$ is the line bundle with a holomorphic
section $s_{Z}$ cutting out $Z,$ i.e. $Z=\{s_{z}=0\}.$ 
\begin{cor}
\label{cor:ldp for vanishing}The laws of the normalized empirical
measure of the determinantal point process associated to $H_{kZ}$
satisfy a LDP on $X-Z$ with a good rate functional whose unique minimizer
is compactly supported on $X-Z.$\end{cor}
\begin{proof}
The map $\Psi_{k}\mapsto\Psi_{k}/s_{Z}$ establishes a unitary isomorphism
between the Hilbert space $H_{kZ}$ above and the vector space $(H^{0}(X,k(L-L_{Z}))$
with the inner product induced by the volume form $\nu$ and the (non-continuous)
metric $\left\Vert \dot{}\right\Vert e^{\log|s_{Z}|}$ on the line
bundle $L-L_{Z}$ over $X$ (the metric blows up along $Z).$ Since,
$\nu$ satisfies the BM-property wrt any domain in $X$ (as is proved
just as in the proof of the previous corollary) the LDP follows from
the previous theorem applied to the line bundle $L-L_{Z}$ and the
set $U.$ 
\end{proof}
Of course, the previous LDP holds for much more general measures $\nu.$
For example we can take $\nu=fdV_{X}$ where $f$ is continuous and
positive on $X-Z.$ Then we just have to assume that the corresponding
inner products on $H_{k}$ are finite for $k>>1.$ The proof also
applies to the more general setting when $Z$ is replaced by a simple
normal crossings divisor \[
Z:=t_{1}Z_{1}+\cdots t_{m}Z_{m}\]
 where $t_{i}\geq0$ if we let $H_{k}$ be defined by taking the vanishing
order on $Z_{i}$ to be the round-down of $kt_{i}.$ For example,
when $Z=t_{1}Z_{1}$ and $s_{Z_{1}}$ is a section of $L$ then we
can take $t_{1}\in[0,1[$ so that $t_{1}=0$ corresponds to the situation
in the previous sections where $H_{kZ}$ is the full Hilbert space
$H^{0}(X,kL)$ and as $t_{1}\rightarrow1$ the scaled dimension of
$H_{kZ}$ tends to zero. Hence, physically $t_{1}$ plays the role
of the {}``filling fraction'' familiar from the Quantum Hall effect
when $n=1$ (where $X$ is the Riemann sphere $\P^{1}$ and $Z_{1}$
is the point at infinity). This latter case has also been studied
extensively from the point of view of Laplacian growth (for example
in connection to the Hele-Shaw flow); see \cite{za,h-m} and references
therein. Here we just briefly point out the relation to Laplacian
growth on a Riemann surface of arbitrary genus: 
\begin{prop}
Fix a point $Z$ on a compact Riemann surface $(X,\omega)$ equipped
with a real two-form $\omega$ such that $\int_{X}\omega=1.$ For
$t\in]0,1[$ the equilibrium measure associated to $(X,\omega,tZ)$
(i.e. the minimizer appearing in \ref{cor:ldp for vanishing}) may
be written as \begin{equation}
\mu_{t}:=1_{D_{t}}\omega,\label{eq:reg of equi meas}\end{equation}
for a closed set $D_{t}.$ The right derivative of $\mu_{t}$ exists
and its value at $t=t_{0}$ coincides with the equilibrium measure
of the (un-) weighted compact set $(D_{t_{0}},0).$ In particular,
if $D_{t_{0}}$ is a domain with piecewise smooth boundary, then this
latter measure is supported on $\partial D_{t_{0}}$ and its density
is the normal derivative of the Green function of $(D_{t_{0}},0)$
with a pole at $Z.$\end{prop}
\begin{proof}
The regularity statement can be deduced from the general $\mathcal{C}^{1.1}-$regularity
result of envelopes in \cite{berm1}. Next, it will be convenient
to switch to the weight notation, i.e. we let $\omega$ be the curvature
form for a weight $\Phi$ on a line bundle of degree $L$ (see section
\ref{sub:Notation-and-general}) and we write $\Phi_{t}$ for the
upper envelope of all psh weights $\chi_{t}$ on $L$ such that $\chi_{t}\leq\Phi$
on $X$ and $\chi_{t}\leq t\log|s|^{2}+C_{\chi_{t}}$ for some constant
$C_{\chi_{t}}.$ We let $D_{t}$ be the closed subset where $\chi_{t}=\Phi.$
Then $\mu_{t}$ is the curvature current of $\chi_{t},$ i.e. $\mu_{t}=dd^{c}\chi_{t}.$
Moreover, if follows immediately from the definition that $\Phi_{t+h}\leq\Phi_{t}$
if $h>0$ and that $t\mapsto\Phi_{t}$ is concave in $t.$ Accordingly,
$D_{t+h}\subset D_{t}$ and by concavity the right derivate $\dot{\Phi}_{t}$
is a decreasing limits of psh weights on $L$ and hence it exists
and is itself a psh weight on $L.$ Moreover, by definition, the right
derivative of $\mu_{t}$ is equal to $dd^{c}\dot{\Phi}_{t}.$ To see
the relation to Green functions we note that on $U:=X-Z$ we can use
$s$ as a trivializing section of $L$ and let $u_{t}:=\log|s|^{2}-\Phi_{t}$
be the description of the weight $\Phi_{t}$ in this trivialization,
denoting by $g_{t}$ its right derivative. It follows from $D_{t+h}\subset D_{t}$
that $\dot{g}_{t}=0$ on $D_{t}$, $dd^{c}g_{t}=0$ in the exterior
and $g_{t}$ has a pole at infinity in $U$ (using that by definition
$\log|s|^{2}-C\leq\dot{\Phi}_{t}\leq\log|s|^{2}+C).$ Hence, $g_{t_{0}}$is
the desired Green function and in the regular case we can use Stokes
theorem to get $\mu_{t_{0}}=dd^{c}g_{t_{0}}=[\partial D_{t_{0}}]\wedge d^{c}g_{t_{0}}$
which gives the required normal derivative on the boundary.
\end{proof}
In other words the proposition says (in the regular case) that the
one-parameter family of domains defined by the supports of the equilibrium
measures $\mu_{t}$ are decreasing in $t$ and their evolution is
driven by the normal derivative of the corresponding Green functions,
which is the definition of Laplacian growth. It would be interesting
to know if an analogous theory of {}``Monge-Ampère growth'' can
be developed in higher dimensions? In another direction it was recently
shown in \cite{r-w} that, in any dimension, the partial Legendre
transform in $t$ of the corresponding equilibrium potentials define
a weak geodesic ray in the (closure) of the space of all Kähler potentials
on $L,$ equipped with the Mabuchy metric.

\subsection{\label{sub:Proof-of-the-non-det ldp}A gen\label{sub:A-generalized-LDP beta}erelized
LDP for $\beta-$ensembles }

Given a weighted measure $(\nu,\omega)$ and a sequence $(\beta_{k})\subset\R_{+}$
we obtain a sequence of random point process on $X$, defined by the
following probability measure on $X^{N_{k}}:$ \[
\mu_{\beta_{k}}^{(N_{k})}:=\frac{1}{\mathcal{Z}_{N_{k},\beta_{k}}}\left\Vert \det\Psi_{k}\right\Vert ^{\beta_{k}}(\nu(x_{1})\otimes\cdots\otimes\nu(x_{N})\]

\begin{thm}
\label{thm:ldp for general beta}Suppose that $\beta_{k}\leq C$ and
$\beta_{k}k\rightarrow\infty$ (in particular, $\beta_{k}\equiv\beta$
is allowed) and that the measure $\nu$ satisfies the \emph{strong}
B-M-property wrt the non-pluripolar set $K$ (i.e. \ref{eq:bm strong}
holds for any $\omega(=\omega_{\phi})$. Then the random point processes
above satisfy a LDP with the same rate functional as in Theorem \ref{thm:intro large dev},
but with the speed $\beta_{k}kN_{k}.$\end{thm}
\begin{proof}
We will first show that \begin{equation}
\mathcal{F}_{N_{k},\beta_{k}}[\phi]:=\frac{1}{_{kN_{k}}}\log\left\Vert \det(\Psi_{k})e^{-k\phi}\right\Vert _{L^{\beta_{k}}(X^{N_{k}},\nu^{\otimes N_{k}})}\rightarrow-\mathcal{E}_{\omega}(P_{(K,\omega)}\phi)\label{eq:conv of free energy for beta}\end{equation}
 as $k\rightarrow\infty,$ where we have taken $\Psi_{k}$ to be defined
wrt a base $(\Psi_{i,k})$ which is orthonormal wrt the Hermitian
metric determined by $(\nu,\left\Vert \cdot\right\Vert ).$ To this
end first note that for any $\epsilon>0$ there is a constant $C_{\epsilon}$
such that \[
\sup_{K^{N_{k}}}\left\Vert \det(\Psi_{k})e^{-k\phi}\right\Vert \leq C_{\epsilon}^{N_{k}}e^{\epsilon\beta_{k}kN_{k}}\left\Vert \det(\Psi_{k})e^{-k\phi}\right\Vert _{L^{\beta_{k}}(X^{N_{k}},\nu^{\otimes N_{k}})}\]
Indeed, this follows immediately from applying the inequality \ref{eq:bm strong}
$N_{k}$ times, one time for each variable of $\psi(x_{1},...,x_{N_{k}}):=\frac{1}{k}\log\left\Vert \det(\Psi_{k})e^{-k\phi}\right\Vert $
and with $p=\beta_{k}k.$ But then the first point in Theorem \ref{thm:[asym-Fekete]}
gives that \[
\mathcal{F}_{N_{k},\beta_{k}}[\phi]\leq0-\mathcal{E}_{\omega}(P_{(K,\omega)}\phi)+o(1)\leq\mathcal{F}_{N_{k},\beta_{k}}[\phi]+\epsilon\beta_{k}+\frac{1}{\beta_{k}k}\log C_{\epsilon}\]
 which finishes the proof \ref{eq:conv of free energy for beta}.
Finally, since \[
\frac{1}{\beta_{k}kN_{k}}\log\E(e^{-\beta_{k}k(\phi(x_{1})+...+)})=\mathcal{F}_{N_{k},\beta_{k}}[\phi]-\mathcal{F}_{N_{k},\beta_{k}}[0]\rightarrow-\mathcal{E}_{\omega}(P_{(K,\omega)}\phi)+\mathcal{E}_{\omega}(P_{(K,\omega)}0)\]
the abstract Ellis-Gärtner theorem gives the desired LDP (alternatively,
the direct proof in section \ref{sub:A-direct-proof} can also be
used to conclude the proof of the LDP, just as in the determinantal
case). 
\end{proof}

\subsection{Proof of Cor \ref{cor:non-sharp tail intro} }

Given $\lambda>0$ and $u\in C^{0}(X)$ we let $F_{\lambda}$ be the
set of all probability measures $\mu$ on $K$ such that $\int_{X}\phi(\mu-\mu_{eq})\geq\lambda.$
Since this is a compact set, not containing $\mu_{eq},$ and since
the rate functional $H(=H_{(K,\omega)})$ is good it follows immediately
that $\inf_{\mu\in F_{\lambda}}H=C_{\lambda}>0$ and hence the corollary
is a consequence of the upper bound contained in Theorem \ref{thm:intro large dev}.
Next, let us show that when $\mu_{eq}=\omega$ we have \[
\inf_{\mu\in F_{\lambda}}E_{\omega}(\mu)=\frac{2\lambda^{2}}{\left\Vert d\phi\right\Vert _{X}^{2}}\]
(to simplify the notation we will assume that $V=1).$To this end
we first note that \[
E(\mu)=-\frac{1}{2}\int_{:X}u_{\mu}dd^{c}u_{\mu}=\frac{1}{2}\int_{:X}du_{\mu}\wedge d^{c}u_{\mu}:=\frac{1}{2}(u_{\mu},u_{\mu})\]
 where $u_{\mu}$ is satisfies $dd^{c}u_{\mu}=\mu-\omega.$ Since,
$(\cdot,\cdot)$ defines a positive definite inner product on the
Dirichlet space $H:=\{v:dv\in L^{2}(X)\}/\R$ and since $F_{\lambda}=\{\mu=\omega+dd^{c}v:\,\,(\phi,v)\geq\lambda\}$
it follows immediately from the Cauchy-Schwartz inequality that $E(\mu)=\frac{2\lambda^{2}}{\left\Vert d\phi\right\Vert _{X}^{2}}.$
The proof is concluded by noting that (in any dimension) we have when
$\omega\geq0$ that $H_{(X,\omega)}=E_{\omega}.$ Indeed, $C(X,\omega):=\inf_{\mu\in\mathcal{P}(X)}E_{\omega}(\mu)=\mathcal{E}_{\omega}(P_{(X,\omega)}0)=\mathcal{E}_{\omega}(0)=0.$

\section{\label{sec:Relation-to-bosonization}Relation to bosonization and
effective field theories}

We will conclude with a heuristic discussion about some relations
to the notion of bosonization (or fermion-boson correspondence) in
the physics literature (see \cite{st}). We will compare the present
setting with the one in the paper \cite{b-v-} which concerns the
case of a Riemann surface $X$ (but see also \cite{v-v}). Another
useful reference on field theory linking mathematical and physical
terminology is \cite{gaw}. As established in \cite{b-v-} (see also
\cite{v-v}) there is a correspondence between a certain theory of
fermions $\Psi$ on one hand and bosons $\phi$ on the other, on a
Riemann surface $X.$ In the present context $\Psi$ is a complex
spinor coupled to the line bundle $L$ and $\phi$ is a smooth function
on $X.$ The main ingredient in the correspondence is the equality
\begin{equation}
\left\langle \left\Vert \Psi(x_{1})\right\Vert ^{2}\cdots\left\Vert \Psi(x_{N})\right\Vert ^{2}\right\rangle =\left\langle e^{i\phi(x_{1})}\cdots e^{i\phi(x_{N})}\right\rangle ,\label{eq:boson ans}\end{equation}
 (see formula $4.4$ in \cite{b-v-}) where the brackets denote integration
against (formal) functional integral measures of the form $\mathcal{D}\Psi\mathcal{D}\bar{\Psi}e^{-S_{Ferm}(\Psi,\bar{\Psi})}$
and $\mathcal{D}\phi e^{-S_{bose}(\phi)},$ respectively. The fermionic
action $S_{Ferm}$ has a standard form (see below) and the problem
is to find a bosonic action $S_{bose}$ so that the ansatz \ref{eq:boson ans}
above holds.

Comparing with the geometric setup in the present paper we will, in
this section, consider the case when the compact set $K$ is all of
$X$ and equip $L$ with a smooth metric with (normalized) curvature
form $\omega,$ that will however not assumed to be positive. We will
also fix an Hermitian metric on $X$ with volume form $dV.$ In the
previous terminology we will hence consider the determinantal point
process associated to the weighted measure $(dV,\omega)$ on $X.$

Let us now explain how the LDP in Theorem \ref{thm:intro large dev}
(and in particular its variant in Theorem \ref{thm:(Effective-bosonization)-Let}
) can be interpreted as an asymptotic/effective version of the boson-fermion
correspondence on a complex manifold of arbitrary dimension $n$ with
the choice \begin{equation}
S_{bose}(\phi)=-\frac{1}{(-i)^{n-1}}\mathcal{E}_{-i\omega}(\phi)\label{eq:ansats for bos act}\end{equation}
where it will, in this section, be convenient to use a new normalization
of the functional $\mathcal{E}_{\omega}:$ \[
\mathcal{E}_{\omega}(\phi)=\frac{1}{(n+1)!}\sum_{j=0}^{n}\int_{X}\omega_{\phi}^{j}\wedge(\omega)^{n-j}\]
where we recall that $\omega_{\phi}:=\omega+dd^{c}\phi(=\omega+\frac{i}{2\pi}\partial\bar{\partial}\phi)$
(and where we have, for simplicity, assumed that the volume $V$ of
$L$ is one). With this normalization the following $n+1-$homogeneity
holds: \begin{equation}
\mathcal{E}_{c\omega}(c\psi)=c^{n+1}\mathcal{E}_{\omega}(\psi)\label{eq:homogen}\end{equation}
In particular, when $n=1$ we get \[
S_{bose}(\phi)=\mathcal{E}_{-i\omega}(\phi)=-\frac{1}{2}\int_{X}d\phi\wedge d^{c}\phi-i\int_{X}\phi\omega\]
which in the notation of \cite{b-v-} corresponds precisely to the
decomposition \[
S_{bose}(\phi)=S_{1}+S_{2}\]
 given in \cite{b-v-}, in the case when $X$ has genus zero. In the
higher genus case soliton type terms are added in \cite{b-v-} to
the action, which describe the topological sector of the bosonic theory.
This is related to the fact that the field $\phi$ should really be
assumed to be circle valued (i.e. its values are only well-defined
up to integer periods). But as explained in \cite{b-v-} one may assume
that $\phi$ is single-valued when dealing with the non-topological
sector - compare the discussion in connection to formula 3.19 in \cite{b-v-}
(anyway the topological sector can be shown to give a lower order
contribution in the large $k$ limit studied below)

To make the connection with the LDP we will, as before, consider the
the limit when $L$ is replaced by the large tensor power $kL.$ Since,
$N_{k}\sim k^{n}$ this is also the limit of many particles and the
asymptotic boson-correspondence will hence give a bosonic field theory
description of a collective theory of fermions. As explained in the
introduction of the paper we will consider {}``clouds'' of points
$(x_{1},...,x_{N})$ that can be described by a {}``macro state'',
i.e. a limiting continuous distribution, or more precisely a measure
$\mu=\rho dV$ on $X:$ \begin{equation}
\frac{1}{N_{k}}\sum_{i=1}^{N_{k}}\delta_{x_{i}}\approx\mu\label{eq:appr measure}\end{equation}
 in a suitable smeared out sense.

\subsection{The fermionic side}

First recall the representation of the Slater determinant \ref{eq:slater det}
(expressed in an orthonormal base wrt to $(dV,\omega))$ as a functional
integral over Grassman (anti-commuting) fields (compare formula 4.5
in \cite{b-v-} ):

\begin{equation}
\left\Vert (\det\Psi)(x_{1},...,x_{N})\right\Vert ^{2}=C_{N}\int\mathcal{D}\Psi\mathcal{D}\bar{\Psi}e^{-S_{ferm}(\Psi,\bar{\Psi})}\left\Vert \Psi(x_{1})\right\Vert ^{2}\cdots\left\Vert \Psi(x_{1})\right\Vert ^{2},\label{eq:fermion path int}\end{equation}
integrating of over all complex (Dirac) spinors $\Psi$, i.e. over
all smooth sections of the exterior algebra $\Lambda^{0,*}(T^{^{*}}X)\otimes L.$
Here $S_{ferm}(\Psi,\bar{\Psi})$ is the fermionic action \[
S_{ferm}(\Psi,\bar{\Psi})=\int_{X}\left\langle \D_{A}\Psi,\Psi\right\rangle dV,\]
 expressed in terms of the Dirac operator $\D_{A}$ on $\Lambda^{0,*}(T^{^{*}}X)\otimes L$
induced by the complex structure on $X$ and $L$ and coupled to the
gauge field/connection $A$ whose curvature is $-i2\pi\omega$ and
to the Hermitian metric on $X$ with volume form $dV.$ Concretely,
we have $\D_{A}=\overline{\partial}+\overline{\partial}^{*}$ expressed
in terms of the adjoint of the $\overline{\partial}$- operator. For
the Riemann surface case see \cite{b-v-}. This latter case, i.e.
when $n=1,$ is special since $S_{ferm}(\Psi,\bar{\Psi})$ is then
independent of the choice of Hermitian metric on $X$ (i.e. the action
is conformally invariant). Indeed, decomposing $\Psi=\Psi_{0}+\Psi_{1}$
gives \[
S_{ferm}(\Psi,\bar{\Psi})=i\int_{X}\overline{\partial}\Psi_{0}\wedge\overline{\Psi_{1}}+\Psi_{0}\wedge\overline{\partial\Psi_{1}}\]
where the integrand is naturally a $(1,1)-$form and may hence be
integrated over $X$ (to simplify the formula we have assumed that
$L$ is trivial above - in general one also has to use the metric
on $L$ or equivalently couple $\overline{\partial}$ to the gauge
field $A).$ The integer $N$ appearing in \ref{eq:fermion path int}
is the dimension of the space of zero-modes of $\D_{A}$ on $\Lambda^{0,*}(T^{^{*}}X)\otimes L$
which coincides with $H^{0}(X,L)$ since we have assumed that $L$
is ample (more precisely, this will be true once we replace $L$ with
$kL$ for $k$ sufficiently large due to Kodaira vanishing in positive
degrees). Moreover, the constant $C_{N}$ in \ref{eq:fermion path int}
is the invers of the Ray-Singer analytic torsion of the complex $\Lambda^{0,*}(T^{^{*}}X)\otimes L$
(compare section \ref{sub:Analytic-torsion,-exponentially}), if we
use zeta function regalization of the corresponding formal determinants. 

Now applying the LDP in Theorem \ref{thm:(Effective-bosonization)-Let}
for the Slater determinants multiplied by the analytic torsion and
using \ref{eq:fermion path int} and \ref{eq:appr measure} hence
gives \[
\left\langle \left\Vert \Psi(x_{1})\right\Vert ^{2}\cdots\left\Vert \Psi(x_{N})\right\Vert ^{2}\right\rangle \sim e^{-kN_{k}E_{\omega}(\mu)}\]

\subsection{The bosonic side}

We will use an argument involving analytic continuation in a real
parameter $t$ (which will be set to $-i$ in the end). To this end
we consider the path integral \[
\int\mathcal{D}\phi e^{-S_{t}(\phi)}e^{t\phi(x_{1})}\cdots e^{t\phi(x_{N})},\,\,\,\,-S_{t}(\phi)=t^{-(n-1)}\mathcal{E}_{t\omega}(\phi)\]
Next, we will show that, in the limit when $L$ gets replaced with
$kL$ and $\omega$ with $k\omega$ we have \begin{equation}
\int\mathcal{D}\phi e^{-S_{t}(\phi)}e^{t\phi(x_{1})}\cdots e^{t\phi(x_{N})}\sim e^{t^{2}k^{n+1}E_{\omega}(\mu)}\label{eq:asympt of bos t-integr}\end{equation}
Accepting this for the moment we see that the effective bosonization
explained above follows by invoking analytic continuation and setting
$t=-i.$ 

To see how the asymptotics above come about first note that we get
the following exponent in the integral, after setting $\phi=k\psi:$ 

\[
t^{-(n-1)}\mathcal{E}_{tk\omega}(k\psi)-tk(\psi(x_{1})+...\psi(x_{N}))\]
Using the $n+1-$homogeneity \ref{eq:homogen} the previous expression
may be written as \[
k^{(n+1)}(t^{-(n-1)}\mathcal{E}_{t\omega}(\psi)-t\frac{1}{k^{n}}(\psi(x_{1})+...\psi(x_{N})))\]
and hence, using \ref{eq:appr measure}, it can be approximated by
\[
k^{(n+1)}(t^{-(n-1)}\mathcal{E}_{t\omega}(\psi)-t\int_{X}\psi\mu\]
Now we get \[
\int\mathcal{D}\phi e^{-S_{t}(\phi)}e^{t\phi(x_{1})}\cdots e^{t\phi(x_{N})}\sim e^{k^{n+1}\sup_{\psi}(t^{-(n-1)}\mathcal{E}_{t\omega}(\psi)-t\int_{X}\psi\mu)}\]
Denote by $\psi_{t}$ a function where the sup is above is attained.
Since it is a stationary point we get the equation \[
t^{-(n-1)}\frac{1}{n!}(t\omega+dd^{c}\psi_{t})^{n+1}=t\mu\Leftrightarrow\frac{1}{n!}(t\omega+dd^{c}\psi_{t})^{n+1}=t^{n}\mu\]
(see the variational properties in Prop \ref{pro:of energy-1}). Hence
a solution is obtained by setting $\psi_{t}=t\psi_{\mu}$ where $\psi_{\mu}$
is a potential for $\mu$ (solving the equation for $t=1).$ Finally,
since\[
(t^{-(n-1)}\mathcal{E}_{t\omega}(t\psi_{\mu})-t\int_{X}(t\psi_{\mu})\mu)=t^{2}(\mathcal{E}_{\omega}(\psi_{\mu})-\int_{X}\psi_{\mu}\mu)=t^{2}E_{\omega}(\mu)\]
(using the homogeneity \ref{eq:homogen} for $c=t$ in the first step
and the formula \ref{eq:energy in terms of pot} for $E_{\omega}(\mu).$
This proves the asymptotics \ref{eq:asympt of bos t-integr} up to
a subtle point that was neglected in the previous argument: when $n>1$
the function $\psi_{t}$ may not be a maximizer even though it is
a stationary point. Equivalently, this means that the potential $\phi_{\mu}$
of $\mu$ may not maximize the functional \[
\phi\mapsto\mathcal{E}_{\omega}(\phi)-\int_{X}\phi\mu\]
 over the\emph{ whole} space $C^{\infty}(X).$ By Theorem \ref{thm:var sol of ma}
it is a maximizer on the subspace $\mathcal{H}_{\omega}$ of $C^{\infty}(X)$
where $\omega_{\phi}\geq0,$ but it is known that the functional above
is not even bounded from above on the whole space $C^{\infty}(X).$
Coming back to the original variable $\phi$ this problem could have
been circumvented by restricting the integration to all $\phi$ such
that $dd^{c}\phi\geq-kt\omega.$ This can be seen as a regularization
similar to frequency cut-offs usually used in quantum field theory.
Alternatively, one could expect that there is a choice of action \[
S_{bos}=-\frac{1}{(-i)^{n-1}}\mathcal{E}_{-i\omega}+S'\]
 where the term $S'$ has the effect of localizing the integral to
the subspace $\mathcal{H}_{\omega/k}$ in the large $k$ limit and
that $S'$ will be lower order in $k$ and hence negligible on the
space $\mathcal{H}_{\omega/k}.$ It would be very interesting to find
a model (when $n>1)$ where such a mechanism could be analyzed.

\subsection{Consistency with the central limit theorem}

Instead of going further into the subtle points raised above we will
just observe that the choice of bosonic action \ref{eq:ansats for bos act}
is consistent with the central limit theorem (CLT) proved in \cite{berm3}.
We will consider the case when the curvature form $\omega$ is strictly
positive. Then the CLT referred to above says, in physical terms,
that the (suitably scaled) fluctuations of the empirical measure \ref{eq:intro random measure}
converge to a random measure/charge-distribution for a statistical
Coulomb gas ensemble described by the Boltzmann weight \[
e^{-\tilde{E}(\nu)}\mathcal{D}v\]
where $\nu(=\rho dV)$ is a signed measure (i.e. a difference of two
positive measures) such that $\int_{X}\nu=0$ and \[
\tilde{E}(\nu)=-\frac{1}{2}\int_{X}(\Delta u_{\nu})u_{\nu}dV(=\frac{1}{2}\left\Vert du_{\nu}\right\Vert _{X}^{2}),\,\,\,\,\nu=\Delta u_{\nu}\]
where $\Delta$ is the Laplacian taken wrt a Riemannian metric $g$
on $X$ and $dV$ is the volume form of $g.$ Equivalently, the Green
potential $u_{\nu}$ is a\emph{ massless boson field} on $X$ (i.e.
a \emph{Gaussian free field} in mathematical terms). Using Fourier
transforms (see \cite{berm3} and references therein) the precise
mathematical meaning of the convergence is \begin{equation}
\int_{X^{N}}dV(x_{1})...\rho^{(N)}(x_{1},.,,,x_{N})e^{\frac{1}{a_{k}}iu(x_{1})}\cdots e^{iu(x_{N})}\rightarrow e^{-\frac{1}{2}\left\Vert du\right\Vert _{X}^{2}}\label{eq:clt conv phy}\end{equation}
where $\rho^{(N)}$ is the Slater determinant appearing above and
$a_{k}=k^{(n-1)/2}.$ To obtain the latter convergence from the previous
bosonization ansatz we approximate the lhs above by \begin{equation}
C_{N_{k}}\int\mathcal{D}\mu e^{\frac{k^{n}}{a_{k}}i\int u(\mu-\frac{\omega^{n}}{n!})}D\phi e^{S_{bose,k}(\phi)+k^{n}i\int\phi\mu}\label{eq:clt using phys}\end{equation}
 Note that when $k=1$ we may expand \[
S_{bose,1}(\phi)=-\frac{1}{(-i)^{n-1}}\mathcal{E}_{-i\omega}(\phi)=-i\int_{X}\phi\frac{\omega^{n}}{n!}-\frac{1}{2}\int_{X}d\phi\wedge d^{c}\phi\wedge\frac{\omega^{n-1}}{(n-1)!}+...\]
where the dots indicate a sum of $n-1$ terms of the form \[
\frac{1}{2}\int_{X}d\phi\wedge d^{c}\phi\wedge(dd^{c}\phi)^{j}\wedge\omega^{n-j},\,\,\,\, j\geq1\]
which hence is of order $\geq3$ in $\phi$ (some of the coefficients
will be imaginary). In the limit when we are changing $L$ by $kL$
and $\omega$ by $k\omega$ we write $\phi=k\psi$ as before, so that
\[
S_{bose,k}(\phi)=k^{n+1}S_{bose,1}(\psi)\]
and hence the exponent in the $\phi$ integral in \ref{eq:clt using phys}
may be written as \[
k^{n+1}(i\int\psi(\mu-\frac{\omega^{n}}{n!})-\frac{1}{2}\left\Vert d\psi\right\Vert _{X}^{2}+...)\]
Les us now make the following change of variables: \[
\mu=\frac{\omega^{n}}{n!}+\frac{1}{k^{(n+1)/2}}\nu,\,\,\,\psi=\frac{1}{k^{(n+1)/2}}v\]
Then the previous expression may be written as \[
i\int v\nu-\frac{1}{2}\left\Vert dv\right\Vert _{X}^{2}+O(k^{-1})\]
and hence the integral over $D\phi$ may be approximated by a Gaussian
integral over $v$ which, as usual, may be evaluated as \[
e^{-\frac{1}{2}\left\Vert dv_{\nu}\right\Vert _{X}^{2}}\]
All in all this means that the integral \ref{eq:clt using phys} may
be approximated by \[
\int\mathcal{D}\nu e^{^{i\int u\nu}}e^{-\frac{1}{2}\left\Vert dv_{\nu}\right\Vert _{X}^{2}}\]
and performing the Gaussian integration again finally gives the end
result $e^{-\frac{1}{2}\left\Vert du\right\Vert _{X}^{2}},$ thus
confirming \ref{eq:clt conv phy}.

\end{document}